\documentclass[12pt,reqno]{amsart}
\usepackage{amsmath,amsopn,amssymb,amsthm,multicol}
\usepackage{hyperref}
\usepackage{graphicx}

\newcommand{\fr}{\mathfrak}

 \newtheorem{lemma} {Lemma} [section]
\newtheorem{theorem}[lemma]{Theorem} 
\newtheorem{remark}[lemma] {Remark} 
\newtheorem{prop} [lemma]{Proposition}  
\newtheorem{definition}[lemma] {Definition} 
 
\newtheorem{example}[lemma] {Example}

\newtheorem{assum}[lemma]{Assumption}
\newtheorem{no}[lemma]{Notation}
\numberwithin{equation}{section}

\begin{document}

\title[Geodesics as deformations of one-parameter subgroups]{Geodesics as deformations of one-parameter subgroups in homogeneous manifolds}
\author{Nikolaos Panagiotis Souris} \thanks{The author is supported by the EPSRC grant EP/R008205/1.} 
\address{University of Reading, Department of Mathematics and Statistics, Reading RG6 6AX, UK}
\email{n.souris@reading.ac.uk}

\maketitle
\medskip

\begin{abstract}

We solve explicitly the geodesic equation for a wide class of (pseudo)-Riemannian homogeneous manifolds $(G/H,\mu)$, including those with $G$ compact, as well as non-compact semisimple Lie groups, under a simple algebraic condition for the metric $\mu$.  We prove that these manifolds are geodesically complete and their geodesics are orbits in $G/H$ of a product of $N$ one-parameter subgroups of $G$, where $N\in \mathbb N$.  This is intimately related to the fact that the metrics $\mu$ can be regarded as $N$-parameter deformations of a Riemannian metric in $G/H$ with special symmetries.  We show that there exists a wealth of metrics having the aforementioned type of geodesics; those metrics are constructed from any Lie subgroup series of the form $H\subset H_1\subset \cdots \subset H_{N-1}\subset G$.  
\medskip


\medskip

\noindent

 
 \end{abstract}

\section{Introduction}

 An open problem in geometry is the construction of pseudo-Riemannian metrics whose geodesics can be explicitly described (\cite{Bur-Mat}).  Owing to the complexity of the geodesic equation $\nabla_{\dot{\gamma}}{\dot{\gamma}}=0$, metrics with explicitly known geodesics are scarce even in manifolds of the simplest nature, namely homogeneous manifolds.  For the homogeneous case, the most common metrics with known geodesics are those whose geodesics are orbits of one-parameter groups of isometries.  Such examples are the standard Riemannian metrics in symmetric spaces (\cite{Kob-No} p. 231) as well as the metrics induced from bi-invariant metrics on Lie groups.  On the other hand, given a general pseudo-Riemannian metric $\mu$ in a homogeneous manifold $G/H$, there has been rather limited insight on what are the possible forms of the geodesics or which geometric aspects of $(G/H,\mu)$ determine those forms.\\
In this paper we determine explicitly the geodesics for a large class of homogeneous manifolds by using the following general approach:  We firstly fix a Riemannian metric $\mu_0$ in $G/H$ whose geodesics are orbits of one-parameter subgroups $\phi_X(t)$ of $G$.  Then we view any other invariant metric $\mu$ in $G/H$ as an $N$-parameter deformation of $\mu_0$, where $N$ is the number of distinct eigenvalues of the unique operator $A_{\mu}:T_{eH}(G/H)\rightarrow T_{eH}(G/H)$ given by 

\begin{equation}\label{nstar}\mu( X,Y)=\mu_0(A_{\mu}X,Y), \quad X,Y \in T_{eH}(G/H).\end{equation} 

  Subsequently, we prove that, under an algebraic condition for the eigenspaces of $A_{\mu}$, the one-parameter subgroup geodesics of $\mu_0$ deform to orbits of a product of $N$ one-parameter subgroups, yielding the geodesics of the metric $\mu$.\\
In particular we choose $\mu_0$ to be \emph{naturally reductive} (see Definition \ref{defnat}) which is a well-studied class of metrics with additional symmetries (see \cite{Da-Nic}), \cite{Da-Zi}, \cite{Go}).  When $G$ is compact, an example of such a metric is the metric induced from a bi-invariant metric in $G$.  The central result of the paper can be summarised as follows.  

\begin{theorem}\label{main1} Let $G/H$ be a homogeneous manifold that admits a $G$-invariant Riemannian naturally reductive metric $\mu_0$.  Let $\mu$ be a $G$-invariant metric in $G/H$ and let $A_{\mu}$ be the corresponding endomorphism of the tangent space $T_{eH}(G/H)$, given by Equation (\ref{nstar}).  Assume that the eigenspaces $\fr{m}_1,\dots,\fr{m}_N$ of $A_{\mu}$ satisfy the relations

\begin{equation}\label{integg}[\fr{m}_i,\fr{m}_j]\subseteq \fr{m}_i\quad \makebox{if} \quad i<j.\end{equation}    

 Then for any geodesic $\gamma$, with $\gamma(0)=eH$, there exist $N$ vectors $X_1,\dots, X_N$ in the Lie algebra $\fr{g}$ of $G$ such that 

\begin{equation*}\label{GF}\gamma(t)=\phi_{X_1}(t)\cdots \phi_{X_N}(t)H, \quad \makebox{for any} \quad t\in \mathbb R,\end{equation*}

\noindent where $\phi_{X_i}$ denotes the one-parameter subgroup of $G$ generated by $X_i$, $i=1,\dots,N$.  In particular, $(G/H,\mu)$ is geodesically complete.
\end{theorem}

The vectors $X_i$, $i=1,\dots N$, depend on $\dot{\gamma}(0)$ and on the eigenvalues of $A_{\mu}$;  they are described explicitly in the detailed statement of Theorem \ref{main1} in Section \ref{mainres}.\\  
Theorem \ref{main1} yields the geodesics for an extensive class of homogeneous manifolds.  Indeed, manifolds admitting a Riemannian naturally reductive metric $\mu_0$ include, among other examples, any manifold $G/H$ with $G$ compact, any non-compact and connected semisimple Lie group as well as any simply connected Heisenberg group (see Section \ref{natredmet}).  Subsequently, any such manifold admits a plethora, depending on the maximality of $H$ in $G$, of pseudo-Riemannian metrics $\mu$ satisfying Condition (\ref{integg}).  In particular, we show that $N$-parameter families of metrics $\mu$ satisfying Condition (\ref{integg}) can be constructed from any Lie subgroup series of the form $H\subset H_1\subset \cdots \subset H_{N-1}\subset G$ ( Theorem \ref{main3}).  If the subgroups $H_i$ are closed, these metrics can be regarded as $N$-parameter deformations of the naturally reductive metric $\mu_0$ along the manifolds $H_{i+1}/H_i$.  Two-parameter examples include \emph{Cheeger deformation metrics} induced by a homogeneous fibration $G/H\rightarrow G/K\rightarrow K/H$ (see Section \ref{appl}).  Furthermore, there exist numerous examples of spaces whose geodesics with respect to any $G$-invariant metric are given by Theorem \ref{main1} (see Section \ref{appl}).\\ 
The proof of Theorem \ref{main1} relies on two key ingredients.  The first is an algebraic formulation of the geodesic equation for $(G/H,\mu)$ that we obtain in Theorem \ref{main2}.  It generalises a Lax pair expression for the Euler equation of motion in compact Lie groups which was obtained by Arnold in \cite{Arn}.  We derive it by using the additional symmetries of the endomorphism $A_{\mu}$ that are induced from the naturally reductive metric $\mu_0$.  The second ingredient is a special family of curves $T_i^j:\mathbb R\rightarrow \operatorname{Aut}(\fr{g})$, $1\leq i\leq j\leq N$, that we introduce in Section \ref{nhomo}.  The technical properties of the family $T_i^j$ allow a more convenient handling of the aforementioned algebraic expression of the geodesic equation for the curves $\gamma(t)=\phi_{X_1}(t)\cdots \phi_{X_N}(t)H$.\\
For $N=1$, Theorem \ref{main1} recovers the fact that the geodesics with respect to any scalar multiple of a naturally reductive metric $\mu_0$ are orbits of single one-parameter subgroups, while for $N=2$ it recovers the main results in the papers \cite{Arv-So} and \cite{Do}.\\

\subsection{Structure of the paper}  The paper is structured as follows:\\
  In Section \ref{prel} we set the notation of the paper and we recall some preliminary facts about homogeneous manifolds.\\
In Section \ref{natredmet} we collect some basic facts about naturally reductive metrics.\\
  In Section \ref{mainres} we state the main results of the paper in the order that they will be proved:  Theorem \ref{main2} which states the algebraic formulation of the geodesic equation, Theorem \ref{main11} which is a more detailed statement of the central result \ref{main1}, and Theorem \ref{main3} which establishes the existence of metrics with geodesics of the form $\gamma(t)=\phi_{X_1}(t)\dots \phi_{X_N}H$ in manifolds admitting a Riemannian naturally reductive metric.\\
    In Section \ref{equiv} we study the geodesic equation $\nabla_{\dot{\gamma}}{\dot{\gamma}}=0$ for homogeneous spaces.  In Proposition \ref{main33} we obtain equivalent expressions for the equation.  By using Proposition \ref{main33} as well as the additional symmetries of naturally reductive metrics, we prove Theorem \ref{main2}.\\
      Section \ref{nhomo} serves as preparatory work for the proof of the central Theorem \ref{main11}.  In particular, we introduce a special family of curves $T_i^j:\mathbb R\rightarrow \operatorname{Aut}(\fr{g})$ and we derive several properties of this family in Lemma \ref{egal}.  By using the curves $T_i^j$, we express the geodesic equation of Theorem \ref{main2} explicitly for curves of the form $\gamma(t)=\phi_{X_1}(t)\cdots \phi_{X_N}(t)H$ (Proposition \ref{th3}).\\
      In Section \ref{proof} we use the aforementioned explicit expression in order to prove the central Theorem \ref{main11} of the paper.  The contribution of the structural condition (\ref{integg}) to the proof is highlighted in Lemma \ref{claim2}.\\
        In Section \ref{appl} we give a proof of the existence result, Theorem \ref{main3}, by using a characterisation of Kostant for naturally reductive metrics.  Within the proof, we describe the explicit construction of $N$-parameter metrics whose geodesics are orbits of a product of $N$ one-parameter subgroups.  We construct those metrics by considering Lie subgroup series of the form $H\subset H_1\subset \cdots \subset H_{N-1}\subset G$.  We also give examples of such metrics.  Finally, based on a result of Ziller, we give an interpretation of Theorem \ref{main11} for special cases of two-parameter metrics.

 \bigskip
\noindent\textit{Acknowledgments.}
The author is supported by the EPSRC grant EP/R008205/1.

  \section{Preliminaries and notation}\label{prel}
The main references for this section are the books \cite{Hel}, \cite{Kob-No} and \cite{On}.

\subsection{Adjoint representations and one-parameter subgroups of Lie groups}

 Let $G$ be a Lie group and let $\fr{g}$ be its Lie algebra.  We denote by $L_g,R_g:G\rightarrow G$ the left and right translations by $g\in G$ respectively.  Let $\operatorname{Ad}:G\rightarrow \operatorname{Aut}(\fr{g})$ be the adjoint representation of $G$ given by $\operatorname{Ad}(g)X=d(L_g\circ R_{g^-1})_eX$, for $g\in G$ and $X\in \fr{g}$, and let $\operatorname{ad}:\fr{g}\rightarrow \operatorname{End}(\fr{g})$ be the adjoint representation of $\fr{g}$ given by $\operatorname{ad}(X)Y=(d\operatorname{Ad})_e(X)Y=[X,Y]$, for $X,Y\in \fr{g}$.  Any $\operatorname{Ad}$-invariant object in $\fr{g}$ is also $\operatorname{ad}$-invariant and the converse is true if $G$ is connected.\\
  For $X\in \fr{g}$ we denote by $\phi_X:\mathbb R \rightarrow G$ the one-parameter subgroup of $G$ generated by $X$, that is $\phi_X(t)=\exp(tX)$.  We recall the identities
  
  \begin{eqnarray}\label{id0}\phi_X(t)^{-1}&=&\phi_X(-t),\quad \makebox{for any} \quad X\in \fr{g},\quad t\in \mathbb R,\\
\label{id1}\phi_X(t+s)&=&\phi_X(t)\phi_X(s),\quad \makebox{for any} \quad X\in \fr{g},\quad t,s\in \mathbb R \quad \makebox{and}\\
\label{id2}g\phi_X(t)g^{-1}&=&\phi_{\operatorname{Ad}(g)X}(t),\quad \makebox{for any} \quad g\in G,\quad X\in \fr{g},\quad t\in \mathbb R.\end{eqnarray}

\subsection{Homogeneous spaces and $G$-invariant metrics} 
Let $M$ be a homogeneous manifold.  For any Lie group $G$ acting transitively on $M$, the manifold $M$ is diffeomorphic to the \emph{homogeneous space} $G/H$, where $H$ is the isotropy subgroup of the point $eH$ called the \emph{origin} of $G/H$.

\begin{assum}\label{asum}  For the rest of the paper we assume that the action of $G$ on $G/H$ is effective.  Equivalently, any subgroup of $H$ that is normal in $G$ is trivial.
\end{assum}

 Assumption \ref{asum} is not restrictive;  for any subgroup $K$ of $H$ that is normal in $G$, we can consider the diffeomorphic space $(G/K)/(H/K)$ without losing crucial geometric information for $G/H$, as any $G$-invariant object in $G/H$ is also $G/K$-invariant and vice versa.\\
  Let $\tau_g:G/H\rightarrow G/H$ be \emph{the left translation in $G/H$} by $g\in G$, given by $\tau_g(g^{\prime}H)=(gg^{\prime})H$, for $g^{\prime}H \in G/H$.  Any left translation $\tau_g$ is a diffeomorphism of $G/H$ and its inverse is $\tau_{g^{-1}}$.  A pseudo-Riemannian metric $\mu$ in $G/H$ is called \emph{$G$-invariant} if the left translations $\tau_g$ are isometries of $(G/H,\mu)$.\\  
Let $\pi:G\rightarrow G/H$ be the natural projection.  Then $\pi$ is a submersion.  Moreover, 

\begin{equation}\label{interc}\tau_g\circ \pi=\pi\circ L_g,\quad \makebox{for any} \quad g\in G.\end{equation}

\subsection{Reductive homogeneous spaces}\label{klgries}
  We denote by $\fr{g}$, $\fr{h}$ the Lie algebras of $G$ and $H$ respectively.  The space $G/H$ is called \emph{reductive} if the Lie algebra $\fr{g}$ admits a direct sum decomposition
 
 \begin{equation}\label{decom}\fr{g}=\fr{h}\oplus \fr{m},\end{equation}
 
 \noindent such that $\operatorname{Ad}(H)\fr{m}\subseteq \fr{m}$.  In that case, $\fr{m}$ is identified with the tangent space $T_{eH}(G/H)$ via the differential of $\pi$, that is $T_{eH}(G/H)=d\pi_e(\fr{g})=\fr{m}$.  Any space $G/H$ that admits a $G$-invariant Riemannian metric is reductive.\\
In any reductive space $G/H$, the $G$-invariant pseudo-Riemannian metrics $\mu$ are in bijection with the $\operatorname{Ad}(H)$-invariant, symmetric and non-degenerate bilinear forms $\langle \ ,\ \rangle$ on $\fr{m}$.  Moreover, if $\mu$ is Riemannian then $\langle \ ,\ \rangle$ is positive definite.  If  $\langle \ ,\ \rangle$ is an $\operatorname{Ad}(H)$-invariant form on $\fr{m}$, then any operator $\operatorname{ad}(X)$, $X\in \fr{h}$, is skew-symmetric with respect to $\langle \ ,\ \rangle$.

\subsection{The metric endomorphism}
We fix a $G$-invariant Riemannian metric $\mu_0$ in $G/H$.  Let $\mu$ be a $G$-invariant pseudo-Riemannian metric in $G/H$ and consider the $\operatorname{Ad}(H)$-invariant forms $\langle \ ,\ \rangle_0$ and $\langle \ ,\ \rangle$ on $\fr{m}$ corresponding to the metrics $\mu_0$ and $\mu$ respectively.  There exists a unique endomorphism $A_{\mu}:\fr{m}\rightarrow \fr{m}$ such that 
  
  \begin{equation}\label{metrop}\langle X,Y\rangle =\langle A_{\mu}X,Y\rangle_0,\quad \makebox{for any} \quad X,Y\in \fr{m}.\end{equation} 

We call $A_{\mu}$ the \emph{metric endomorphism of $\mu$}.  Clearly, $A_{\mu_0}=\operatorname{Id}$.  The metric endomorphism is \emph{symmetric with respect to $\langle \ ,\ \rangle_0$}, \emph{it has non-zero eigenvalues} because $\langle \ ,\ \rangle$ is non-degenerate, and it is \emph{$\left.\operatorname{Ad}(H)\right|_{\fr{m}}$-equivariant}, that is $A_{\mu}\circ \left.\operatorname{Ad}(h)\right|_{\fr{m}}=\left.\operatorname{Ad}(h)\right|_{\fr{m}}\circ A_{\mu}$, for any $h\in H$.  The last property is induced from the $\operatorname{Ad}(H)$-invariance of $\langle \ ,\ \rangle$.  Conversely, any endomorphism with the aforementioned properties corresponds to a unique $G$-invariant pseudo-Riemannian metric in $G/H$.  Moreover, the metric $\mu$ is Riemannian if and only if the eigenvalues of $A_{\mu}$ are positive.  Since $A_{\mu}$ is diagonalisable, the tangent space $\fr{m}$ admits a $\langle \ ,\ \rangle_0$-orthogonal decomposition

\begin{equation*}\fr{m}=\fr{m}_1\oplus \cdots \oplus \fr{m}_N,\end{equation*}

\noindent into the eigenspaces $\fr{m}_i$, $i=1,\dots,N$, of $A_{\mu}$.  As a result, the spaces $\fr{m}_i$ are $\operatorname{Ad}(H)$-invariant.  We refer to \cite{So} for the explicit form of $A_{\mu}$ when $G$ is compact. \\
Let $\lambda_1,\dots,\lambda_N$ be the eigenvalues of $A_{\mu}$ corresponding to the eigenspaces $\fr{m}_1,\dots,\fr{m}_N$ respectively.  We will use the notation

\begin{equation}\label{lerdak}A_{\mu}=\lambda_1\left.\operatorname{Id}\right|_{\fr{m}_1}+\cdots +\lambda_N\left.\operatorname{Id}\right|_{\fr{m}_N},\end{equation}

 \noindent for $A_{\mu}$.  By letting $\lambda_1,\dots,\lambda_N$ vary over $\mathbb R^*$, the endomorphisms (\ref{lerdak}) define an $N$-parameter family of metrics $\mu=\mu(\lambda_1,\dots,\lambda_N)$ which can be regarded as an $N$-parameter family of deformations of the metric $\mu_0$.

\section{Naturally reductive metrics}\label{natredmet}

In this section we recall the definition and several basic facts regarding naturally reductive metrics. 

\begin{definition}\label{defnat}\emph{A $G$-invariant metric $\mu_0$ in $G/H$ is called} naturally reductive \emph{(equivalently the space $(G/H,\mu_0)$ is called naturally reductive) if $G/H$ admits a decomposition (\ref{decom}) such that the corresponding $\operatorname{Ad}(H)$-invariant form $\langle \ ,\ \rangle_0$ in $\fr{m}$ satisfies the identity}

\begin{equation}\label{natred}\langle [X,Y]_{\fr{m}},Z\rangle_0+\langle Y,[X,Z]_{\fr{m}}\rangle_0=0,\end{equation}

\noindent \emph{for any $X,Y,Z\in \fr{m}$.  We will also say that \emph{$\mu_0$ is naturally reductive with respect to the decomposition (\ref{decom}).}}
\end{definition}

Examples of naturally reductive spaces include \emph{irreducible symmetric spaces}, as well as \emph{normal homogeneous spaces}, that is spaces $G/H$ endowed with a metric induced from a bi-invariant metric in $G$.  We recall the following.
  
  \begin{prop}\label{oneiln} \emph{(\cite{On} p. 313)} Let $M$ be a homogeneous manifold endowed with a $G$-invariant naturally reductive metric $\mu_0$.  Then the geodesics of $(M,\mu_0)$ are orbits in $M$ of one-parameter subgroups of $G$.  
  \end{prop} 
  
  \subsection{Naturally reductive metrics induced by $\operatorname{Ad}$-invariant forms}\label{cf}
  Any $\operatorname{Ad}$-invariant, non-degenerate, symmetric bilinear form $Q$ in $\fr{g}$, which is positive-definite in $\fr{h}^{\bot}$, gives rise to a Riemannian naturally reductive metric in $G/H$ as follows:  Since $Q$ is positive-definite in $\fr{h}^{\bot}$, there is a $Q$-orthogonal direct sum decomposition $\fr{g}=\fr{h}\oplus \fr{m}$, where $\fr{m}=\fr{h}^{\bot}$.  We use the $\operatorname{Ad}$-invariance of $Q$ to write 
  
  \begin{equation*}Q(\operatorname{Ad}(H)\fr{m},\fr{h})=Q(\fr{m},\operatorname{Ad}(H)\fr{h})\subseteq Q(\fr{m},\fr{h})=\{0\}.\end{equation*}

Hence $\operatorname{Ad}(H)\fr{m}\subseteq \fr{m}$, and we can in turn identify $\fr{m}$ with $T_{eH}(G/H)$.  Moreover, the inner product $\langle \ ,\ \rangle_0:=\left.Q\right|_{\fr{m}\times \fr{m}}$ is $\operatorname{Ad}(H)$-invariant, therefore it defines a $G$-invariant Riemannian metric $\mu_0$ in $G/H$.  Since $Q$ is $\operatorname{Ad}$-invariant, any operator $\operatorname{ad}(X)$, $X\in \fr{m}$, is skew-symmetric with respect to $Q$.  By taking into account the $Q$-orthogonality between $\fr{m}$ and $\fr{h}$, for any $X,Y,Z\in \fr{m}$ we have

\begin{equation*}\langle [X,Y]_{\fr{m}},Z\rangle_0+\langle Y,[X,Z]_{\fr{m}}\rangle_0=Q([X,Y],Z)+Q(Y,[X,Z])=0,\end{equation*}
\noindent which, by virtue of Definition \ref{defnat}, implies that the metric $\mu_0$ is naturally reductive.  This leads us to the following definition.

\begin{definition}\label{indic}\emph{We will say that a Riemannian naturally reductive metric $\mu_0$ on $G/H$ \emph{is induced from an $\operatorname{Ad}$-invariant form in $\fr{g}$} if there exists an $\operatorname{Ad}$-invariant, non-degenerate, symmetric bilinear form $Q$ in $\fr{g}$, which is positive-definite in $\fr{m}=\fr{h}^{\bot}$, and such that the corresponding (to the metric $\mu_0$) $\operatorname{Ad}(H)$-invariant inner product $\langle \ ,\ \rangle_0$ in $\fr{m}$ is equal to $\left.Q\right|_{\fr{m}\times \fr{m}}$.}
\end{definition}

  According to a result of Kostant (see \cite{Kos}, \cite{Da-Zi}), any Riemannian naturally reductive metric $\mu_0$ in $G/H$ is induced from an $\operatorname{Ad}$-invariant form in a suitable Lie subalgebra of $\fr{g}$.  We will state the result of Kostant in Section \ref{appl}, where it will be used for proving Theorem \ref{main3}.

\subsection{Manifolds admitting Riemannian naturally reductive metrics}
Although the classification of naturally reductive spaces is open, several partial classification results have been obtained.  We mention them briefly below, providing the corresponding references:\\

\noindent (a) Any manifold $G/H$ with $G$ compact admits at least one Riemannian naturally reductive metric $\mu_0$.  Here $\mu_0$ is the induced metric in $G/H$ from any $\operatorname{Ad}$-invariant inner product on $\fr{g}$.  The Riemannian naturaly reductive metrics in simple compact Lie groups were classified in \cite{Da-Zi}.\\
\noindent (b) Any connected, non-compact semisimple Lie group admits a Riemannian naturally reductive metric $\mu_0$.  Here $\mu_0$ is $G\times K$-invariant, where $K$ is a connected subgroup of $G$ generated from a maximal compact subalgebra $\fr{k}$ of $\fr{g}$.  The metric is given as follows.  Let $Q$ be the Killing form of $\fr{g}$ and consider the Cartan decomposition $\fr{g}=\fr{k}\oplus \fr{p}$.  Then the metrics $\mu_0$ defined by the $\operatorname{Ad}(H)$-invariant form 

\begin{equation*}\lambda_1\left.Q\right|_{\fr{k}\times \fr{k}}+\lambda_2\left.Q\right|_{\fr{p}\times \fr{p}}\end{equation*} 

\noindent are naturally reductive (\cite{Hal-Ian}).  Moreover, the Riemannian naturally reductive metrics in spaces $G/H$, with $G$ connected, non-compact and semisimple, were classified in \cite{Go}.\\
\noindent (c) Any Riemannian left-invariant metric in a simply connected Heisenberg group is naturally reductive (\cite{Go}).  The naturally reductive nilmanifolds were classified in \cite{Go}, where it was also proved that any naturally reductive nilmanifold is at most two-step nilpotent.\\
\noindent (d) The naturally reductive spaces up to dimension six were classified in \cite{Ag-Fe-Fr}.

\section{Main results}\label{mainres}

We will state the main results of the paper.

\subsection{Algebraic formulation of the geodesic equation}

Consider a curve $\gamma:J\subseteq \mathbb R\rightarrow G/H$ with $0\in J$.  Since $\pi:G\rightarrow G/H$ is a submersion, there exists a smooth lift $\alpha:J\rightarrow G$ of $\gamma$ on $G$, that is $\gamma=\pi\circ \alpha$.  Let $\omega:J\rightarrow \fr{g}$ be the lift of $\alpha$ on $\fr{g}$ via the Maurer-Cartan form, that is

\begin{equation}\label{omega}\omega(t):=dL_{\alpha^{-1}(t)}(\dot{\alpha}(t)),\quad t\in J.\end{equation} 

Consider the decomposition $\fr{g}=\fr{h}\oplus\fr{m}$, and denote by $\omega_{\fr{m}}$ the $\fr{m}$-component of $\omega$.  Let $\mu$ be a $G$-invariant pseudo-Riemannian metric in $G/H$, let $A_{\mu}:\fr{m}\rightarrow \fr{m}$ be the corresponding metric endomorphism, and let $F:J\rightarrow \fr{m}$ be the curve defined by 

\begin{equation}\label{M}F(t):=A_{\mu}\omega_{\fr{m}}(t), \quad t\in J.\end{equation} 

The geodesic equation for $\gamma$ reduces to an algebraic equation for the curves $\omega$ and $F$ as given by the following theorem.

\begin{theorem}\label{main2} Let $G/H$ be a homogeneous space that admits a $G$-invariant Riemannian naturally reductive metric $\mu_0$ with respect to the decomposition $\fr{g}=\fr{h}\oplus \fr{m}$, and let $\mu$ be a $G$-invariant pseudo-Riemannian metric in $G/H$.  Let $\gamma=\pi\circ \alpha:J\rightarrow G/H$ be a curve through the origin $eH$ in $G/H$, and let $\omega:J\rightarrow \fr{g}$ and $F:J\rightarrow \fr{m}$ be the curves given by equations (\ref{omega}) and (\ref{M}) respectively.  Then $\gamma$ is a geodesic if and only if 

\begin{equation}\label{lax}\bigg(\frac{d}{dt}F(t)=[F(t),\omega(t)]\bigg) \makebox{$(\mathrm{mod}$ $\fr{h})$},\quad \makebox{for any} \quad t\in J.\end{equation}

Further assume that $\mu_0$ is induced from an $\operatorname{Ad}$-invariant form in $\fr{g}$ in the sense of Definition \ref{indic} (e.g. when $G$ is compact).  Then $\gamma$ is a geodesic if and only if 

\begin{equation}\label{lax1}\frac{d}{dt}F(t)=[F(t),\omega(t)],\quad \makebox{for any} \quad t\in J.\end{equation}

\end{theorem}

When $G$ is compact, $H=\{e\}$ and $\mu$ is Riemannian, Theorem \ref{main2} recovers Theorem 1 in \cite{Arn}.  Theorem \ref{main2} will be proved in Section \ref{equiv}.

\subsection{Geodesics as deformations of one-parameter subgroups}\label{intep}

The following is the complete statement of the central result of the paper.

\begin{theorem}\label{main11}Let $G/H$ be a homogeneous space that admits a $G$-invariant Riemannian naturally reductive metric $\mu_0$ with respect to the decomposition $\fr{g}=\fr{h}\oplus \fr{m}$.  Let $\mu$ be a $G$-invariant pseudo-Riemannian metric in $G/H$ and let $A_{\mu}:\fr{m}\rightarrow \fr{m}$ be the corresponding metric endomorphism given by Equation (\ref{metrop}).   Assume that there exists a decomposition $\fr{m}=\fr{m}_1\oplus \cdots \oplus \fr{m}_N$, into eigenspaces of $A_{\mu}$, such that
\begin{equation}\label{integ}[\fr{m}_i,\fr{m}_j]\subseteq \fr{m}_i\quad \makebox{if} \quad i<j,\quad i,j=1,\dots,N.\end{equation}    

  Then any geodesic $\gamma$ of $(G/H,\mu)$, such that $\gamma(0)=eH$, is defined on $\mathbb R$, and there exist $N$ vectors $X_1,\dots, X_N\in \fr{g}$ such that 

\begin{equation}\label{GF}\gamma(t)=\phi_{X_1}(t)\cdots \phi_{X_N}(t)H, \quad \makebox{for any} \quad t\in \mathbb R.\end{equation}

\noindent The vectors $X_i$ are given explicitly as follows:  let $\lambda_1,\dots,\lambda_N$ be the eigenvalues of $A_{\mu}$ corresponding to the eigenspaces $\fr{m}_1,\dots,\fr{m}_N$, and denote by $\pi_i:\fr{m}\rightarrow \fr{m}_i$, $i=1,\dots,N$, the linear projections of $\fr{m}$ onto $\fr{m}_i$.  Then 

\begin{equation}\label{VE}X_1={\frac{1}{\lambda_1}\sum_{k=1}^N\lambda_k\pi_k(\dot{\gamma}(0))}, \quad \makebox{and} \quad X_i=\frac{(\lambda_{i-1}-\lambda_i)}{\lambda_{i-1}\lambda_i}\sum_{k=i}^N{\lambda_k\pi_k(\dot{\gamma}(0))}, \quad i\geq 2.\end{equation}
\end{theorem}

  \begin{remark}The form (\ref{GF}) is invariant by the left translations in $G/H$, therefore Theorem \ref{main11} implies that all geodesics in $G/H$ (and not only those through $eH$) are orbits of a product of $N$ one-parameter subgroups.  Indeed, under the assumptions of Theorem \ref{main11}, and given that the left translations are isometries, any geodesic $\widetilde{\gamma}$ through a point $gH\in G/H$ is given by 
  
  \begin{equation*}\widetilde{\gamma}(t)=\tau_g(\phi_{X_1}(t)\cdots \phi_{X_N}(t)H)=(g\phi_{X_1}(t)g^{-1})\cdots (g\phi_{X_N}(t)g^{-1})(gH),\end{equation*} which, by virtue of identity (\ref{id2}), is equal to $\phi_{\operatorname{Ad}(g)X_1}(t)\cdots \phi_{\operatorname{Ad}(g)X_N}(t)(gH)$.   Hence $\widetilde{\gamma}$ is also an orbit of a product of $N$ one-parameter subgroups.
  
  \end{remark}
  

 Theorem \ref{main11} is proved in Section \ref{proof}.

\subsection{Existence of metrics satisfying Theorem \ref{main11}}  For any homogeneous manifold $M$ admitting a Riemannian naturally reductive metric, the following theorem establishes the existence of metrics $\mu$ in $M$ such that the geodesics of $(M,\mu)$ have the form (\ref{GF}). 

\begin{theorem}\label{main3} Let $M$ be a homogeneous manifold that admits a Riemannian naturally reductive metric.  Then there exists a homogeneous space $G/H$, diffeomorphic to $M$, with the following property:  For any $N\in \mathbb N$ and for any $N-tuple$ of Lie groups $(H_1,\dots,H_N)$ such that

\begin{equation*}H\subset H_1\subset \cdots \subset H_{N-1}\subset H_N:=G,\end{equation*}

\noindent there exists an $N$-parameter family of $G$-invariant pseudo-Riemannian metrics\\ $\mu=\mu(\lambda_1,\dots,\lambda_N)$, where $\lambda_1,\dots,\lambda_N\in \mathbb R^*$, such that $(G/H,\mu)$ satisfies the conditions of Theorem \ref{main11}.  In particular, the geodesics of $(G/H,\mu)$ are given by Equation (\ref{GF}).\\
Moreover, the aforementioned property holds for any homogeneous space $G/H$ with $G$ compact.
\end{theorem} 

The proof of Theorem \ref{main3}, along with the explicit construction of the metrics \\$\mu(\lambda_1,\dots,\lambda_N)$, will be given in Section \ref{appl}.

\section{The geodesic equation in homogeneous spaces}\label{equiv}

In this section, we derive several equivalent expressions for the geodesic equation in pseudo-Riemannian homogeneous spaces $(G/H,\mu)$ (Proposition \ref{main33}).  Subsequently, by using those expressions, we prove Theorem \ref{main2}.  We refer the reader to \cite{Khe-Mi} for a Hamiltonian formulation of the geodesic equation in Riemannian homogeneous spaces.\\
Firstly, we recall the Killing vector fields in $G/H$ induced by vectors in $\fr{g}$:  for any $W\in \fr{g}$, the correspondence $W^*:G/H\rightarrow T(G/H)$ given by 

\begin{equation}\label{kil}W^*_{gH}=\left.\frac{d}{dt}\right|_{t=0}\phi_W(t)gH=d(\pi\circ R_g)W,\quad gH\in G/H,\end{equation}

\noindent is a well-defined vector field in $G/H$.   Moreover, it is a \emph{Killing vector field with respect to any $G$-invariant metric in $G/H$}.  To verify the last assertion, we observe that the local flows of $W^*$ are the functions $f_t:U\subseteq G/H\rightarrow G/H$ given by $f_t(gH)=\phi_W(t)gH$, for $gH\in U$;  the flows are isometries with respect to any $G$-invariant metric because $df_t=d\tau_{\phi_W(t)}$ and the left translations are isometries.\\
  Since $\pi$ is a submersion, there exists an orthonormal frame in the tangent bundle $T(G/H)$ that consists of vector fields of the form $W^*$, $W\in \fr{g}$.  Therefore, 
\begin{equation}\label{span}T_{gH}(G/H)=\operatorname{span}_{\mathbb R}\{W_{gH}^*:W\in \fr{g}\},\quad \makebox{for any $gH\in G/H$}.\end{equation}
 
 We will use the following notation.

\begin{no}We denote by $\fr{X}(G/H)$ the Lie algebra of vector fields in $G/H$.  Let $X,Y,Z\in \fr{X}(G/H)$ and let $\mu$ be a metric in $G/H$.  We denote by $X\mu(Y,Z)$ the function induced from the action of $X$ on the function $\mu(Y,Z):G/H\rightarrow \mathbb R$.  In other words, for $p\in G/H$, $X\mu(Y,Z)(p)$, is the derivative of $\mu(Y,Z)$ in the direction of the vector $X_p$.
\end{no}

We proceed to state the first result of this section.

\subsection{Equivalent expressions for the geodesic equation}

\begin{prop}\label{main33}Let $G/H$ be a homogeneous manifold, let $\mu$ be a $G$-invariant pseudo-Riemannian metric in $G/H$ and let $\gamma=\pi\circ \alpha:J\rightarrow G/H$ be a curve in $G/H$ such that $\gamma(0)=eH$.  The following are equivalent:\\

\noindent (i) The curve $\gamma$ is a geodesic of $(G/H,\mu)$.\\

\noindent (ii) The curve $\gamma$ satisfies the equation

\begin{equation*}\label{equiv1}\quad \dot{\gamma}\mu( W^*,\dot{\gamma})({\gamma(t)})=0,\end{equation*}

\noindent for any $W\in \fr{g}$ and $t\in J$, where $W^*$ is given by Equation (\ref{kil}).\\
Moreover, assume that $G/H$ is reductive, admitting the decomposition $\fr{g}=\fr{h}\oplus \fr{m}$.  Let $\langle \ ,\ \rangle$ be the $\operatorname{Ad}(H)$-invariant form in $\fr{m}$ corresponding to the metric $\mu$, let $\omega:J\rightarrow \fr{g}$ be the lift of $\alpha$ in $\fr{g}$ given by Equation (\ref{omega}), and denote by $\omega_{\fr{m}}(t)$ the $\fr{m}$-component of $\omega(t)$.  Then (i) and (ii) are equivalent to the following:\\

\noindent (iii) The curve $\omega$ satisfies the differential equation

 \begin{equation}\label{geo1}\langle \dot{\omega}_{\fr{m}}(t),W\rangle+\langle \omega_{\fr{m}}(t),[W,\omega(t)]_{\fr{m}}\rangle=0,\end{equation}

\noindent for any $W\in \fr{m}$ and $t\in J$.\\

\end{prop}

In order to prove Proposition \ref{main33} we will need the following lemmas.

\begin{lemma}\label{lem1}

The lift $\omega$ of $\alpha$ on $\fr{g}$, given by Equation (\ref{omega}), satisfies the relation 

\begin{equation*}\omega(t)=-dR_{\alpha(t)}(\dot{\alpha}^{-1}(t)),\quad \makebox{for any} \quad t\in J,\end{equation*}

\noindent where $\dot{\alpha}^{-1}(t)=\frac{d}{dt}{\alpha}^{-1}(t)$.\end{lemma}

 \begin{proof}  For any $t\in J$, Leibniz's product rule in $G$ implies that

\begin{eqnarray*}0&=&\frac{d}{dt}\alpha^{-1}(t)\alpha(t)=\left.\frac{d}{ds}\right|_{s=0}\alpha^{-1}(t+s)\alpha(t)+\left.\frac{d}{ds}\right|_{s=0}\alpha^{-1}(t)\alpha(t+s)\\
&&\\
&=&dR_{\alpha(t)}(\dot{\alpha}^{-1}(t))+dL_{\alpha^{-1}(t)}(\dot{\alpha}(t))=dR_{\alpha(t)}(\dot{\alpha}^{-1}(t))+\omega(t).\quad \quad\quad\quad\quad\quad\quad\qedhere\end{eqnarray*}
\end{proof}
	
\begin{lemma}\label{lem2}
Assume that $G/H$ is reductive, admitting the decomposition $\fr{g}=\fr{h}\oplus \fr{m}$, and let $\langle \ ,\ \rangle$ be an $\operatorname{Ad}(H)$-invariant form in $\fr{m}$.  Let $X,W\in \fr{g}$.  Then
 
 \begin{equation*}\label{eqlem2}\langle X_{\fr{m}},[W,X]_{\fr{m}}\rangle=\langle X_{\fr{m}}, [W_{\fr{m}},X]_{\fr{m}}\rangle.\end{equation*}

\end{lemma}

\begin{proof}
We write $W=W_{\fr{h}}+W_{\fr{m}}$.  Using the fact that $[\fr{h},\fr{h}]\subset \fr{h}$, the $\operatorname{ad}(\fr{h})$-invariance of $\fr{m}$ as well as the skew symmetry of any operator in $\operatorname{ad}(\fr{h})$ with respect to $\langle \ ,\ \rangle$, we obtain  
	
	\begin{eqnarray*}\langle X_{\fr{m}},[W,X]_{\fr{m}}\rangle&=&\langle X_{\fr{m}},[W_{\fr{h}}+W_{\fr{m}},X]_{\fr{m}}\rangle\\
&=&\langle X_{\fr{m}},[W_{\fr{h}},X]_{\fr{m}}\rangle+\langle X_{\fr{m}},[W_{\fr{m}},X]_{\fr{m}}\rangle\\
&=&\langle X_{\fr{m}},[W_{\fr{h}},X_{\fr{m}}]\rangle+\langle X_{\fr{m}},[W_{\fr{m}},X]_{\fr{m}}\rangle\\
&=&\langle X_{\fr{m}},[W_{\fr{m}},X]_{\fr{m}}\rangle.\quad \quad\quad\quad\quad\quad\quad \quad\quad\quad\quad\quad\quad \quad\quad\quad\quad\quad\hspace{3pt}\qedhere
\end{eqnarray*}\end{proof}

We proceed with the proof of Proposition \ref{main33}.\\

\noindent \emph{Proof of Proposition \ref{main33}.}  We consider the Levi-Civita connection $\nabla$ on $T(G/H)$ corresponding to the metric $\mu$.  Denote by $\dot{\gamma}$ a local extension, in $G/H$, of the vector field $\dot{\gamma}(t)$ along the curve $\gamma$.  Then $\gamma$ is a geodesic if and only if $\nabla_{\dot{\gamma}}\dot{\gamma}(\gamma(t))=0$, or equivalently, if and only if

\begin{equation} \label{kozz} \mu (V,\nabla_{\dot{\gamma}}\dot{\gamma})(\gamma(t))=0,\end{equation}

\noindent for any vector field $V\in \fr{X}(G/H)$ and any $t\in J$.  By Koszul's formula (see \cite{On}, p. 61), we have

\begin{equation} \label{koz} \mu (V,\nabla_{\dot{\gamma}}\dot{\gamma})=\dot{\gamma}\mu( V,\dot{\gamma})+\mu(\dot{\gamma},[V,\dot{\gamma}])-\frac{1}{2}V\mu(\dot{\gamma},\dot{\gamma}),\end{equation}

\noindent for any $V\in \fr{X}(G/H)$.  By virtue of relation (\ref{span}) we may assume that 

\begin{equation}\label{sunst}V=W^*,\quad \makebox{for} \quad W\in \fr{g},\end{equation} 

\noindent without any loss of generality.  We also recall the following identities for the Levi-Civita connection.

\begin{eqnarray}\label{identit1}X\mu( Y,Z)=\mu( \nabla_X{Y},Z)+\mu( Y,\nabla_X{Z}),\quad &\makebox{for any}&  \quad X,Y,Z\in \fr{X}(G/H),\\
\label{identit2}\nabla_X{Y}-\nabla_Y{X}=[X,Y],\quad &\makebox{for any}&  \quad X,Y,Z\in \fr{X}(G/H).\end{eqnarray}

Moreover, if $X$ is a Killing vector field then

\begin{equation}\label{identit3}\mu( \nabla_Y{X},Y)=0, \quad \makebox{for any}  \quad Y\in \fr{X}(G/H) \end{equation}

\noindent (see \cite{On}, p. 251).  Under the substitution (\ref{sunst}) and by taking into account identities (\ref{identit1}) - (\ref{identit3}) for $X=W^*$, $Y=Z=\dot{\gamma}$, along with the fact that $W^*$ is a Killing vector field, the sum of the last two terms on the right-hand side of Equation (\ref{koz}) becomes

\begin{eqnarray}
 \mu(\dot{\gamma},[V,\dot{\gamma}])-\frac{1}{2}V\mu(\dot{\gamma},\dot{\gamma})&=&\mu(\dot{\gamma},[W^*,\dot{\gamma}])-\frac{1}{2}W^*\mu(\dot{\gamma},\dot{\gamma})\nonumber\\
 &=&\mu(\dot{\gamma},[W^*,\dot{\gamma}])-\mu( \nabla_{W^*}{\dot{\gamma}},\dot{\gamma})\nonumber \\
&=&\mu( \dot{\gamma}, [W^*,\dot{\gamma}]-\nabla_{W^*}{\dot{\gamma}})\nonumber\\
&=&-\mu( \dot{\gamma}, \nabla_{\dot{\gamma}}{W^*})=0\label{palto}.
\end{eqnarray}

By taking into account Equation (\ref{palto}), equations (\ref{kozz}) and (\ref{koz}) imply that $\gamma$ is a geodesic if and only if 

\begin{equation*}\dot{\gamma} \mu( W^*, \dot{\gamma})(\gamma(t))=0,\end{equation*}

\noindent for any $W\in \fr{g}$, $t\in J$, which verifies the equivalence between (i) and (ii).  Next we will prove the equivalence between (ii) and (iii).  Let $W\in \fr{g}$.  For a fixed $t\in J$ we set 

\begin{equation}\label{Zi}Z(s):=\operatorname{Ad}(\alpha^{-1}(t+s))W\quad \makebox{and} \quad Z:=Z(0).\end{equation}

  By using successively Equation (\ref{kil}), the hypothesis that $\gamma=\pi\circ \alpha$, the invariance of the metric $\mu$ by the left translations $\tau_g:G/H\rightarrow G/H$, the property $df\circ dg=d(f\circ g)$ of the differential, the assumption that $\mu( \ ,\ )_{eH}=\langle \ ,\ \rangle$, identity (\ref{interc}), Equation (\ref{omega}) for $\omega$ as well as relation (\ref{Zi}),  we obtain

\begin{eqnarray}\label{array1}\dot{\gamma} \mu( W^*,\dot{\gamma})(\gamma(t))&=&\left.\frac{d}{ds}\right|_{s=0}\mu\big( W_{\gamma(t+s)}^*,\dot{\gamma}(t+s)\big)_{\gamma(t+s)}\nonumber\\
&&\nonumber\\
&=&\left.\frac{d}{ds}\right|_{s=0}\mu\big( d(\pi\circ R_{\alpha(t+s)})W,d\pi(\dot{\alpha}(t+s))\big)_{\pi(\alpha(t+s))}\nonumber\\
&&\nonumber\\
&=&\left.\frac{d}{ds}\right|_{s=0}\mu\big( d(\tau_{\alpha^{-1}(t+s)}\circ \pi\circ R_{\alpha(t+s)})W,d(\tau_{\alpha^{-1}(t+s)}\circ \pi)(\dot{\alpha}(t+s))\big)_{eH}\nonumber\\
&&\nonumber\\
&=&\left.\frac{d}{ds}\right|_{s=0}\big\langle d(\tau_{\alpha^{-1}(t+s)}\circ \pi\circ R_{\alpha(t+s)})W,d(\tau_{\alpha^{-1}(t+s)}\circ \pi)(\dot{\alpha}(t+s))\big\rangle\nonumber\\
&&\nonumber\\
&=&\left.\frac{d}{ds}\right|_{s=0}\big\langle d(\pi\circ L_{\alpha^{-1}(t+s)}\circ R_{\alpha(t+s)})W,d(\pi\circ L_{\alpha^{-1}(t+s)})(\dot{\alpha}(t+s))\big\rangle\nonumber\\
&&\nonumber \\
&=&\left.\frac{d}{ds}\right|_{s=0}\big\langle d\pi(\operatorname{Ad}(\alpha^{-1}(t+s))W),d\pi (\omega(t+s))\big\rangle\nonumber\\
&&\nonumber\\
&=&\left.\frac{d}{ds}\right|_{s=0}\big\langle d\pi(Z(s)),d\pi (\omega(t+s))\big\rangle\nonumber\\
&&\nonumber\\
&=&\big\langle d\pi(Z),d\pi (\dot{\omega}(t))\big\rangle+\big\langle d\pi\big(\left.\frac{d}{ds}\right|_{s=0}Z(s)\big),d\pi (\omega(t))\big\rangle.
\end{eqnarray}

Subsequently, by using relation (\ref{Zi}), the fact that $\operatorname{Ad}:G\rightarrow \operatorname{Aut}(\fr{g})$ is a homomorphism as well as Lemma \ref{lem1}, we obtain 

\begin{eqnarray*}\left.\frac{d}{ds}\right|_{s=0}Z(s)&=&\left.\frac{d}{ds}\right|_{s=0}(\operatorname{Ad}(\alpha^{-1}(t+s)\alpha(t))Z\\
&=&(d\operatorname{Ad})_e(\left.\frac{d}{ds}\right|_{s=0}\alpha^{-1}(t+s)\alpha(t))Z\\
&=&
(d\operatorname{Ad})_e(dR_{\alpha(t)}(\dot{\alpha}^{-1}(t)))Z\\
&=&
(d\operatorname{Ad})_e(-\omega(t))Z=[Z,\omega(t)].
\end{eqnarray*}

By substituting the above expression into Equation (\ref{array1}), we obtain the equality 

\begin{equation*}\label{ii-iii}\dot{\gamma}\mu( W^*,\dot{\gamma})(\gamma(t))=\big\langle d\pi(Z),d\pi(\dot{\omega}(t))\big\rangle+\big\langle d\pi([Z,\omega(t)]),d\pi(\omega(t))\big\rangle.\end{equation*}

Thus (ii) is equivalent to

\begin{equation}\label{HLDR}\big\langle d\pi(\dot{\omega}(t)),d\pi(Z)\big\rangle+\big\langle d\pi(\omega(t)), d\pi([Z,\omega(t)]\big\rangle=0,\end{equation}

\noindent for any $Z=\operatorname{Ad}(\alpha^{-1}(t))W\in \fr{g}$ and any $t\in J$.  Since $\operatorname{Ad}(\alpha^{-1}(t)):\fr{g}\rightarrow \fr{g}$ is an automorphism, Equation (\ref{HLDR}) is equivalent to 

\begin{equation}\label{ahhh}\big\langle d\pi(\dot{\omega}(t)),d\pi(W)\big\rangle+\big\langle d\pi(\omega(t)), d\pi([W,\omega(t)]\big\rangle=0, \quad \makebox{for any} \quad W\in \fr{g}, \quad t\in J. \end{equation}

\noindent By taking into account the decomposition $\fr{g}=\fr{h}\oplus \fr{m}$, we write $\fr{m}=d\pi(\fr{g})$ and $W_{\fr{m}}=d\pi(W)$, for any $W\in \fr{g}$.  Hence Equation (\ref{ahhh}) can be written as 

\begin{equation}\label{ahu}\big\langle \dot{\omega}_{\fr{m}}(t),W_{\fr{m}}\big\rangle+\big\langle \omega_{\fr{m}}(t), [W,\omega(t)]_{\fr{m}}\big\rangle=0,\quad \makebox{for any} \quad W\in \fr{g},\quad t\in J.\end{equation}

Using Lemma \ref{lem2} we write $\big\langle \omega_{\fr{m}}(t), [W,\omega(t)]_{\fr{m}}\big\rangle=\big\langle \omega_{\fr{m}}(t), [W_{\fr{m}},\omega(t)]_{\fr{m}}\big\rangle$.  Therefore Equation (\ref{ahu}) is equivalent to

\begin{equation}\label{ahur}\big\langle \dot{\omega}_{\fr{m}}(t),W_{\fr{m}}\big\rangle+\big\langle \omega_{\fr{m}}(t),[W_{\fr{m}},\omega(t)]_{\fr{m}}\big\rangle=0,\quad \makebox{for any} \quad W\in \fr{g},\quad t\in J.\end{equation}

Finally, after taking into account that $d\pi:\fr{g}\rightarrow \fr{m}$ is surjective, Equation (\ref{ahur}) is equivalent to

\begin{equation*}\label{ahhhh}\big\langle W,\dot{\omega}_{\fr{m}}(t)\big\rangle+\big\langle \omega_{\fr{m}}(t), [W,\omega(t)]_{\fr{m}}\big\rangle=0,\quad \makebox{for any} \quad W\in \fr{m},\quad t\in J,\end{equation*}

\noindent which concludes the equivalence between (ii) and (iii).\qed\\

We are ready to prove our first main result.

\subsection{Proof of Theorem \ref{main2}}  By virtue of Proposition \ref{main33}, $\gamma$ is a geodesic if and only if Equation (\ref{geo1}) is true.  By taking into account Equation (\ref{metrop}), Equation (\ref{geo1}) is equivalent to 

\begin{equation}\label{zesth}\langle A_{\mu}\dot{\omega}_{\fr{m}}(t),W\rangle_0+\langle A_{\mu}\omega_{\fr{m}}(t),[W,\omega(t)]_{\fr{m}}\rangle_0=0,\end{equation}

\noindent for any $W\in \fr{m}$, $t\in J$.  We write $\omega(t)=\omega_{\fr{m}}+\omega_{\fr{h}}$.  By taking into account the $\operatorname{ad}(\fr{h})$-invariance of $\fr{m}$, the natural reductivity of $\langle \ ,\ \rangle_0$ (Equation (\ref{natred})), the skew-symmetry of any operator in $\operatorname{ad}(\fr{h})$ with respect to $\langle \ ,\ \rangle_0$ as well as the fact that $A_{\mu}\omega_{\fr{m}}(t)\in \fr{m}$, the second term on the left-hand side of Equation (\ref{zesth}) becomes

\begin{eqnarray}\langle A_{\mu}\omega_{\fr{m}}(t),[W,\omega(t)]_{\fr{m}}\rangle_0&=&\langle A_{\mu}\omega_{\fr{m}}(t),[W,\omega_{\fr{m}}(t)]_{\fr{m}}\rangle_0+\langle A_{\mu}\omega_{\fr{m}}(t),[W,\omega_{\fr{h}}(t)]_{\fr{m}}\rangle_0\nonumber\\
&=&
\langle A_{\mu}\omega_{\fr{m}}(t),[W,\omega_{\fr{m}}(t)]_{\fr{m}}\rangle_0+\langle A_{\mu}\omega_{\fr{m}}(t),[W,\omega_{\fr{h}}(t)]\rangle_0\nonumber\\
&=&
-\langle [A_{\mu}\omega_{\fr{m}}(t),\omega_{\fr{m}}(t)]_{\fr{m}},W\rangle_0-\langle [A_{\mu}\omega_{\fr{m}}(t),\omega_{\fr{h}}(t)],W\rangle_0\nonumber\\
&=&
-\langle[A_{\mu}\omega_{\fr{m}}(t),\omega_{\fr{m}}(t)]_{\fr{m}},W\rangle_0-\langle [A_{\mu}\omega_{\fr{m}}(t),\omega_{\fr{h}}(t)]_{\fr{m}},W\rangle_0\nonumber\\
&=&
-\langle [A_{\mu}\omega_{\fr{m}}(t),\omega(t)]_{\fr{m}},W\rangle_0\label{elecpo}.
\end{eqnarray}

By substituting Equation (\ref{elecpo}) into Equation (\ref{zesth}), we deduce that $\gamma$ is a geodesic if and only if
\setlength\abovedisplayskip{10pt}
\begin{equation}\label{ppzesth}0=\langle A_{\mu}\dot{\omega}_{\fr{m}}(t),W\rangle_0-\langle [A_{\mu}\omega_{\fr{m}}(t),\omega(t)]_{\fr{m}},W\rangle_0=\langle A_{\mu}\dot{\omega}_{\fr{m}}(t)-[A_{\mu}\omega_{\fr{m}}(t),\omega(t)]_{\fr{m}},W\rangle_0, \end{equation}

\noindent for any $W\in \fr{m}$, $t\in J$.  Equation (\ref{ppzesth}) is equivalent to 

\begin{equation*}A_{\mu}\dot{\omega}_{\fr{m}}(t)=[A_{\mu}\omega_{\fr{m}}(t),\omega(t)]_{\fr{m}},\quad \makebox{for any} \quad t\in J,\end{equation*}

\noindent which, due to the linearity of $A_{\mu}$, is equivalent to 

\begin{equation*}\bigg(\frac{d}{dt}(A_{\mu}\omega_{\fr{m}}(t))=[A_{\mu}\omega_{\fr{m}}(t),\omega(t)]\bigg)\makebox{$(\mathrm{mod}$ $\fr{h})$}, \quad \makebox{for any} \quad t\in J,\end{equation*}
\noindent yielding the desired Equation (\ref{lax}).\\
To prove the second part of Theorem \ref{main2}, assume that the metric $\mu_0$ is induced from an $\operatorname{Ad}$-invariant form in $\fr{g}$.  In view of Definition \ref{indic}, there exists an $\operatorname{Ad}$-invariant, symmetric and non-degenerate bilinear form $Q$ in $\fr{g}$, and a $Q$-orthogonal direct sum $\fr{g}=\fr{h}\oplus \fr{m}$, such that $\langle \ ,\ \rangle_0=\left.Q\right|_{\fr{m}\times \fr{m}}$.  We need to prove that $\gamma$ is a geodesic if and only if Equation (\ref{lax1}) is true.  By virtue of Equation (\ref{lax}) it suffices to show that $[F(t),\omega(t)]\in \fr{m}$.  To this end, it suffices to show that

\begin{equation}\label{itsuf}[A_{\mu}X_{\fr{m}},X]\in \fr{m},\quad \makebox{for any} \quad X\in \fr{g},\end{equation}  

\noindent which is equivalent to showing that $Q([A_{\mu}X_{\fr{m}},X],a)=0$, for any $X\in \fr{g}$ and $a\in \fr{h}$.  Indeed, by taking into account the skew-symmetry of any operator in $\operatorname{ad}(\fr{g})$ with respect to $Q$, the fact that $A_{\mu}X_{\fr{m}}\in \fr{m}$ along with the $\operatorname{ad}(\fr{h})$-invariance of $\fr{m}$, the $Q$-orthogonality between $\fr{m}$ and $\fr{h}$, the $\operatorname{ad}(\fr{h})$-equivariance of $A$, the symmetry of $A_{\mu}$ with respect to $\left.Q\right|_{\fr{m}\times \fr{m}}$ as well as the fact that $[a,X_{\fr{h}}]\in \fr{h}$, we obtain

\begin{eqnarray*}Q([A_{\mu}X_{\fr{m}},X],a)&=&Q(X,[a,A_{\mu}X_{\fr{m}}])=Q(X_{\fr{m}},[a,A_{\mu}X_{\fr{m}}])\\
&=&Q(X_{\fr{m}},A_{\mu}[a,X_{\fr{m}}])=Q(A_{\mu}X_{\fr{m}},[a,X])\\
&=&-Q([A_{\mu}X_{\fr{m}},X],a),
\end{eqnarray*}

\noindent which implies that $Q([A_{\mu}X_{\fr{m}},X],a)=0$.  Thus relation (\ref{itsuf}) is true and the proof of Theorem \ref{main2} is concluded.

\section{Geodesic orbits of a product of one-parameter subgroups}\label{nhomo}

Consider the curve 

\begin{equation}\label{nhom}\gamma(t)=\phi_{X_1}(t)\cdots \phi_{X_N}(t)H, \quad t\in \mathbb R, \end{equation}

\noindent where $X_i\in \fr{g}$, $i=1,\dots,N$.  In order to simplify the proof of the central Theorem \ref{main11}, we will express the geodesic equation (\ref{lax}) of Theorem \ref{main2} explicitly for curves of the form (\ref{nhom}).  To this end, we introduce the curves $T^j_i:\mathbb R\rightarrow \operatorname{Aut}(\fr{g})$, for $1\leq i\leq j\leq N$, which, upon taking into account identity (\ref{id0}), are defined by 

\begin{equation}\label{Ti}T^j_i(t)=\left\{ 
\begin{array}{lll}\operatorname{Ad}\big(\phi_{X_j}(-t))\cdots \phi_{X_{i+1}}(-t)\big)=\operatorname{Ad}((\prod_{k=i+1}^j\phi_{X_k}(t))^{-1}),\quad \makebox{if} \quad i<j \ \\

\operatorname{Id_{\fr{g}}},\quad \makebox{if}\quad i=j
\end{array}.
\right.
\end{equation}

Then the geodesic equation for the curve (\ref{nhom}) is given by the following proposition.

\begin{prop}\label{th3} Under the assumptions of Theorem \ref{main2}, let $\gamma:\mathbb R\rightarrow G/H$ be the curve (\ref{nhom}).  Then $\gamma$ is a geodesic if and only if the vectors $X_i$, $i=1,\dots,N$, satisfy the equation

\begin{equation}\label{gelan}\bigg(\sum_{i=1}^NA_{\mu}\big[{T_i^N(t)X_i},\sum_{k=i}^N{T_k^N(t)X_k}\big]_{\fr{m}}=\big[A_{\mu}\sum_{i=1}^N({T_i^N(t)X_i)}_{\fr{m}},\sum_{k=1}^N{T_k^N(t)X_k}\big]\bigg)\makebox{$(\mathrm{mod}$ $\fr{h})$},\end{equation}

\noindent for any $t\in \mathbb R$.
\end{prop}

For the proof of Proposition \ref{th3} we will need the next lemma.

 \begin{lemma}\label{egal}
The curves $T_i^j$ defined by relation (\ref{Ti}) satisfy the following properties:

\begin{eqnarray*} &&\makebox{(i)}\quad T_j^m(t)T_i^{j}(t)=T_i^m(t),\quad 1\leq i\leq j\leq m\leq N,\quad t\in \mathbb R.\\
&&\makebox{(ii)}\quad T_i^j(t+s)=\operatorname{Ad}\bigg(\big(\prod_{k=i+1}^j{\phi_{T_k^j(t)X_k}}(s)\big)^{-1}\bigg)T_i^j(t),\quad 1\leq i<j\leq N,\quad t,s\in \mathbb R.\\
&&\makebox{(iii)}\quad \frac{d}{dt}T_i^j(t)X=[T_i^j(t)X,\sum_{k=i}^{j}T_k^j(t)X_k], \quad 1\leq i\leq j\leq N, \quad X\in \fr{g}, \quad t\in \mathbb R.
\end{eqnarray*}

\end{lemma}

 \begin{proof}
 
 Part (i) follows directly from the definition \ref{Ti} of the curves $T_i^j$ and the fact that the map $\operatorname{Ad}:G\rightarrow \operatorname{Aut}(\fr{g})$ is a homomorphism.  In order to prove part (ii), we will use induction on $j>i$.  We recall the identities (\ref{id0}) and (\ref{id1}).  For $j=i+1$, we have that  

\begin{eqnarray*}T_i^j(t+s)&=&T_i^{i+1}(t+s)=\operatorname{Ad}\big((\phi_{X_{i+1}}(t+s))^{-1}\big)\\
&=&
\operatorname{Ad}\big(\phi_{X_{i+1}}(-s)\phi_{X_{i+1}}(-t)\big)=\operatorname{Ad}\big(\phi_{X_{i+1}}(-s)\big)\operatorname{Ad}\big(\phi_{X_{i+1}}(-t)\big)\\
&=&
\operatorname{Ad}\big(\phi_{T_{i+1}^{i+1}(t)X_{i+1}}(-s)\big)T_{i}^{i+1}(t)=\operatorname{Ad}\bigg(\big(\prod_{k=i+1}^{i+1}{\phi_{T_k^{i+1}(t)X_k}}(s)\big)^{-1}\bigg)T_i^{i+1}(t)\\
&=&
\operatorname{Ad}\bigg(\big(\prod_{k=i+1}^j{\phi_{T_k^j(t)X_k}(s)}\big)^{-1}\bigg)T_i^j(t),
\end{eqnarray*}

\noindent which proves the result for the case $j=i+1$.  Next, assume that the induction hypothesis holds for $j=J$.  We set $C(s):=\big(\prod_{k=i+1}^J{\phi_{T_k^J(t)X_k}(s)}\big)^{-1}\in G$.  If $j=J+1$, then by using part (i) of the lemma as well as the induction hypothesis, we obtain 

\begin{eqnarray}\label{Ad1}
T_i^{J+1}(t+s)&=&T_J^{J+1}(t+s)T_i^J(t+s)\nonumber \\
&=&
\operatorname{Ad}\big((\phi_{X_{J+1}}(t+s))^{-1}\big)\operatorname{Ad}\bigg(\big(\prod_{k=i+1}^J{\phi_{T_k^J(t)X_k}(s)}\big)^{-1}\bigg)T_i^J(t)\nonumber \\
&=&
\operatorname{Ad}\big((\phi_{X_{J+1}}(t+s))^{-1}\big)\operatorname{Ad}(C(s))T_i^J(t)\nonumber \\
&=&
\operatorname{Ad}(\phi_{X_{J+1}}(-s))\operatorname{Ad}(\phi_{X_{J+1}}(-t))\operatorname{Ad}(C(s))T_i^J(t)\nonumber \\
&=&
\operatorname{Ad}(\phi_{X_{J+1}}(-s))\operatorname{Ad}(\phi_{X_{J+1}}(-t)C(s))T_i^J(t).
\end{eqnarray}

By using identity (\ref{id2}) for $g:=\phi_{X_{J+1}}(-t)$ as well as part (i) of the lemma, we obtain 

 \begin{eqnarray}\label{Ad2}
\phi_{X_{J+1}}(-t)C(s)&=&gC(s)=g\big(\prod_{k=i+1}^J{\phi_{T_k^J(t)X_k}(s)}\big)^{-1}\nonumber\\
&=&
\bigg(\prod_{k=i+1}^J{g\phi_{T_k^J(t)X_k}(s)g^{-1}}\bigg)^{-1}g=\bigg(\prod_{k=i+1}^J{\phi_{\operatorname{Ad}(g)T_k^J(t)X_k}(s)}\bigg)^{-1}g\nonumber\\ 
&=&
\bigg(\prod_{k=i+1}^J{\phi_{\operatorname{Ad}(\phi_{X_{J+1}}(-t))T_k^J(t)X_k}(s)}\bigg)^{-1}\phi_{X_{J+1}}(-t)\nonumber\\
&=&
\bigg(\prod_{k=i+1}^J{\phi_{T_J^{J+1}(t)T_k^J(t)X_k}(s)}\bigg)^{-1}\phi_{X_{J+1}}(-t)\nonumber\\
&=&
\big(\prod_{k=i+1}^J{\phi_{T_k^{J+1}(t)X_k}(s)}\big)^{-1}\phi_{X_{J+1}}(-t).
\end{eqnarray}

Finally, by substituting Equation (\ref{Ad2}) into Equation (\ref{Ad1}) and by using part (i) of the lemma, we obtain 

\begin{eqnarray*}\label{Ad3}
T_i^{J+1}(t+s)&=&
\operatorname{Ad}(\phi_{X_{J+1}}(-s))\operatorname{Ad}\bigg(\big(\prod_{k=i+1}^J{\phi_{T_k^{J+1}(t)X_k}(s)}\big)^{-1}\bigg)\operatorname{Ad}(\phi_{X_{J+1}}(-t))T_i^J(t)\\
&=&
\operatorname{Ad}\big(\phi_{T_{J+1}^{J+1}(t)X_{J+1}}(-s)\big)\operatorname{Ad}\bigg(\big(\prod_{k=i+1}^J{\phi_{T_k^{J+1}(t)X_k}(s)}\big)^{-1}\bigg)T_J^{J+1}(t)T_i^{J}(t)\\
&=&
\operatorname{Ad}\bigg(\phi_{T_{J+1}^{J+1}(t)X_{J+1}}(-s)\big(\prod_{k=i+1}^J{\phi_{T_k^{J+1}(t)X_k}(s)}\big)^{-1}\bigg)T_i^{J+1}(t)\\
&=&
\operatorname{Ad}\bigg(\big(\prod_{k=i+1}^{J+1}{\phi_{T_k^{J+1}(t)X_k}(s)}\big)^{-1}\bigg)T_i^{J+1}(t),
\end{eqnarray*}

\noindent which concludes the induction and the proof of part (ii) of the lemma.  Part (iii) holds trivially for $i=j$.  For $i<j$, we consider again the curve $C:\mathbb R\rightarrow G$, given by $C(s):=\big(\prod_{k=i+1}^j{\phi_{T_k^j(t)X_k}(s)}\big)^{-1}$.  We have that $C(0)=e$.  Moreover, by using Leibniz's rule for the Lie group product in $G$, we obtain

\begin{equation}\label{C;s}\dot{C}(0)=\left.\frac{d}{ds}\right|_{s=0}\big(\prod_{k=i+1}^j{\phi_{T_k^j(t)X_k}(s)}\big)^{-1}=-\sum_{k=i+1}^j{T_k^j(t)X_k}.\end{equation}

By using part (ii) of the lemma as well Equation (\ref{C;s}), we obtain    

\begin{eqnarray*}\label{lei}\frac{d}{dt}T_i^j(t)X&=&\left.\frac{d}{ds}\right|_{s=0}T_i^j(t+s)X=\left.\frac{d}{ds}\right|_{s=0}\operatorname{Ad}\bigg(\big(\prod_{k=i+1}^j{\phi_{T_k^j(t)X_k}(s)}\big)^{-1}\bigg)T_i^j(t)X\nonumber \\
&=&
\left.\frac{d}{ds}\right|_{s=0}\operatorname{Ad}(C(s))T_i^j(t)X=(d\operatorname{Ad})_e(\dot{C}(0))T_i^j(t)X\\
&=&
\operatorname{ad}\big(-\sum_{k=i+1}^j{T_k^j(t)X_k}\big)T_i^j(t)X=[T_i^j(t)X,\sum_{k=i+1}^{j}T_k^j(t)X_k]\\
&=&
[T_i^j(t)X,\sum_{k=i}^{j}T_k^j(t)X_k],
\end{eqnarray*}

\noindent which concludes the proof of part (iii).\end{proof}

We are now ready to prove Proposition \ref{th3}.\\ 

\noindent \emph{Proof of Proposition \ref{th3}.}  We set 

\begin{equation}\label{periph}\alpha(t)=\phi_{X_1}(t)\cdots \phi_{X_N}(t),\end{equation}

\noindent so that $\gamma=\pi\circ\alpha$.  The proof of the proposition will be concluded once we explicitly compute the quantities $\omega(t)=dL_{\alpha^{-1}(t)}(\dot{\alpha}(t))$ and $\dot{\omega}(t)$, and then substitute the quantities $F(t)=A_{\mu}\omega_{\fr{m}}(t)$ and $\frac{d}{dt}F(t)=A_{\mu}\dot{\omega}_{\fr{m}}(t)$ into the geodesic equation (\ref{lax}).  In particular, it suffices to show that 

\begin{eqnarray}\label{supr1}\omega(t)&=&\sum_{i=1}^N{T_i^N(t)X_i}\quad \makebox{and}\\
\label{supr2}\dot{\omega}(t)&=&\sum_{i=1}^N{\big[T_i^N(t)X_i,\sum_{k=i}^N{T_k^N(t)X_k}\big]}.\end{eqnarray}

We set

\begin{equation*}\alpha_i(s)=(\prod_{k=1}^{i-1}{\phi_{X_k}(t)})\phi_{X_i}(t+s)(\prod_{k=i+1}^{N}{\phi_{X_k}(t)}), \quad i=1,\dots,N.\end{equation*}

By taking into account Equation (\ref{periph}) and by using Leibniz's rule for the Lie group product, we have 

\begin{equation}\label{alder}\omega(t)=dL_{\alpha^{-1}(t)}(\dot{\alpha}(t))=dL_{\alpha^{-1}(t)}(\displaystyle{\left.\frac{d}{ds}\right|_{s=0}\alpha(t+s)})=dL_{\alpha^{-1}(t)}(\sum_{i=1}^N{\left.\frac{d}{ds}\right|_{s=0}\alpha_i(s)}).\end{equation}

If $i<N$, by using the identities (\ref{id1}) and (\ref{id2}) as well as relations (\ref{Ti}), we obtain 

\begin{eqnarray}\label{i<n}
\alpha_i(s)&=&(\prod_{k=1}^{i-1}{\phi_{X_k}(t)})\phi_{X_i}(t+s)(\prod_{k=i+1}^N{\phi_{X_k}(t)})\nonumber \\
&=&
\big(\prod_{k=1}^{i-1}{\phi_{X_k}(t)}\big)\phi_{X_i}(t)\phi_{X_i}(s)\big(\prod_{k=i+1}^N{\phi_{X_k}(t)}\big)\nonumber \\
&=&
(\prod_{k=1}^N{\phi_{X_k}(t)})\big(\prod_{k=i+1}^N{\phi_{X_k}(t)}\big)^{-1}\phi_{X_i}(s)(\prod_{k=i+1}^N{\phi_{X_k}(t)})\nonumber \\
&=&
\alpha(t)\big(\prod_{k=i+1}^N{\phi_{X_k}(t)}\big)^{-1}\phi_{X_i}(s)(\prod_{k=i+1}^N{\phi_{X_k}(t)})\nonumber \\
&=&
\alpha(t)\phi_{\operatorname{Ad}\big((\prod_{k=i+1}^N{\phi_{X_k}(t)})^{-1}\big)X_i}(s)=\alpha(t)\phi_{T_i^N(t)X_i}(s).\end{eqnarray}

Moreover, for $i=N$ it is 

\begin{eqnarray}\label{i=n}
\alpha_N(s)&=&(\prod_{k=1}^{N-1}{\phi_{X_k}(t)})\phi_{X_N}(t+s)=(\prod_{k=1}^{N-1}{\phi_{X_k}(t)})\phi_{X_N}(t)\phi_{X_N}(s)\nonumber \\
&=&
\alpha(t)\phi_{X_N}(s)=\alpha(t)\phi_{T_N^N(t)X_N}(s).
\end{eqnarray}

Equations (\ref{i<n}) and (\ref{i=n}) imply that

\begin{eqnarray}\label{endiam}
\sum_{i=1}^N{\left.\frac{d}{ds}\right|_{s=0}a_i(s)}&=&\sum_{i=1}^N{\left.\frac{d}{ds}\right|_{s=0}(\alpha(t)\phi_{T_i^N(t)X_i}(s))}\nonumber \\
&=&
\sum_{i=1}^N{dL_{\alpha(t)}(T_i^N(t)X_i)}=dL_{\alpha(t)}(\sum_{i=1}^N{T_i^N(t)X_i)}.
\end{eqnarray}

After substituting Equation (\ref{endiam}) into Equation (\ref{alder}), formula (\ref{supr1}) for $\omega(t)$ follows.\\
Finally, formula (\ref{supr2}) is obtained by using part (iii) of Lemma \ref{egal}, along with formula (\ref{supr1}).\qed

\section{Proof of Theorem \ref{main11}}\label{proof}

Let $\gamma:J\rightarrow G/H$ be a geodesic in $(G/H,\mu)$, such that $\gamma(0)=eH$, and let $\hat{\gamma}:\mathbb R\rightarrow G/H$ be the curve 

\begin{equation*}\hat{\gamma}(t)=\phi_{X_1}(t)\cdots \phi_{X_N}(t)H,\quad t\in \mathbb R,\end{equation*}

\noindent where the vectors $X_i$ are given by equations (\ref{VE}).  It suffices to show that $\hat{\gamma}$ is a geodesic with the same initial conditions as $\gamma$.\\
For $i=1,\dots,N$, we set 

\begin{equation*}\label{YI}Y_i:=\pi_i(\dot{\gamma}(0)) \in \fr{m}_i,\quad i=1,\dots,N,\end{equation*}

\noindent so that the equations (\ref{VE}) in the hypothesis of Theorem \ref{main11} become

\begin{equation}\label{VE2}X_1=\frac{1}{\lambda_1}\sum_{k=1}^N{\lambda_kY_k},\quad X_i=\frac{(\lambda_{i-1}-\lambda_i)}{\lambda_{i-1}\lambda_i}\sum_{k=i}^N{\lambda_kY_k}, \quad i\geq 2.\end{equation}

We have $\gamma(0)=\hat{\gamma}(0)=eH$, and equations (\ref{VE2}) imply that

\begin{equation*}\left.\frac{d}{ds}\right|_{t=0}\hat{\gamma}(t)=\sum_{i=1}^NX_i=\sum_{i=1}^NY_i=\sum_{i=1}^N\pi_i(\dot{\gamma}(0))=\dot{\gamma}(0),\end{equation*}  

\noindent that is the curves $\gamma$ and $\hat{\gamma}$ satisfy the same initial conditions.  Therefore, in order to prove Theorem \ref{main11}, it suffices to show that $\hat{\gamma}$ is a geodesic.  To this end, for any $X\in \fr{g}$ and $i=1,\dots, N$, we introduce the curve $T_iX:\mathbb R\rightarrow \fr{g}$ given by 

\begin{equation}\label{ET}T_iX(t):=T_i^N(t)X,\end{equation} 

\noindent where $T_i^N(t)$ are defined by relation (\ref{Ti}).  By virtue of Proposition \ref{th3}, and in view of the curves (\ref{ET}), the curve $\hat{\gamma}$ is a geodesic if the vectors $X_i$ satisfy the equation

\begin{equation}\label{gelanealoga}\sum_{i=1}^NA_{\mu}\big[{T_iX_i(t)},\sum_{k=i}^N{T_kX_k(t)}\big]_{\fr{m}}-\big[A_{\mu}\sum_{i=1}^N({T_iX_i(t))}_{\fr{m}},\sum_{k=1}^N{T_kX_k(t)}\big]=0, \end{equation}

 \noindent for any $t\in \mathbb R$.  For simplicity, we will omit $t$ and write Equation (\ref{gelanealoga}) in terms of the curves $T_iX_i$ as

\begin{equation}\label{gelane}\sum_{i=1}^NA_{\mu}\big[{T_iX_i},\sum_{k=i}^N{T_kX_k}\big]_{\fr{m}}-\big[A_{\mu}\sum_{i=1}^N({T_iX_i)}_{\fr{m}},\sum_{k=1}^N{T_kX_k}\big]=0.\end{equation}

To verify Equation (\ref{gelane}) and thus conclude the proof of Theorem \ref{main11}, we will need two lemmas.  The first lemma will be used in order to express Equation (\ref{gelane}) in terms of the curves $T_iY_i$.  The result of the second lemma relies on Condition (\ref{integ}) in the hypothesis of Theorem \ref{main11}.  It will be used in order to evaluate the metric endomorphism $A_{\mu}$ on the curves $T_iY_i$ which will in turn lead to the verification of Equation (\ref{gelane}). 

\begin{lemma}\label{claim}
For the vectors $X_i$, $i=1,\dots,N$, given by relation (\ref{VE2}), and the curves (\ref{ET}), the following equations are valid.

\begin{eqnarray*}&& \makebox{(i)}\quad T_iX_{i+1}=T_{i+1}X_{i+1},\quad  i=1,\dots,N-1.\nonumber \\
 &&\makebox{(ii)}\quad T_1X_1=\frac{1}{\lambda_1}\sum_{k=1}^N\lambda_k T_kY_k,\quad \makebox{and} \quad T_iX_i=\frac{(\lambda_{i-1}-\lambda_i)}{\lambda_{i-1}\lambda_i}\sum_{k=i}^N{\lambda_kT_kY_k}, \quad i\geq 2.\nonumber \\
&&\makebox{(iii)}\quad \sum_{k=i}^NT_kX_k=\frac{1}{\lambda_{i-1}}\sum_{k=i}^N(\lambda_{i-1}-\lambda_k)T_kY_k,\quad i\geq 2.\nonumber \\
&&\makebox{(iv)}\quad \sum_{k=1}^NT_kX_k=\sum_{k=1}^NT_kY_k.\nonumber 
\end{eqnarray*}
\end{lemma}

\begin{lemma}\label{claim2}

Given Condition (\ref{integ}) in the hypothesis of Theorem \ref{main11} and given the curves (\ref{ET}), the following relations are valid for the vectors $Y_i$, $i=1,\dots,N$.

\begin{eqnarray*}&&\makebox{(i)} \quad T_iY_i(t)\in \fr{m}_i,\quad \makebox{for any} \quad  i=1,\dots,N, \quad t\in \mathbb R.\nonumber \\
&&\makebox{(ii)} \quad [T_iY_i(t),T_kY_k(t)]\in \fr{m}_i,\quad \makebox{for any}\quad  1\leq i\leq k\leq N,\quad t\in \mathbb R.\nonumber
\end{eqnarray*}

\end{lemma}

\noindent \emph{Proof of Lemma \ref{claim}}.  We recall the identity

\begin{equation}\label{id3}\operatorname{Ad}(\phi_X(t))Y=\sum_{k=0}^{\infty}\frac{t^k}{k!}\operatorname{ad}^k(X)Y,\quad \makebox{for any} \quad X,Y\in \fr{g},\quad t\in \mathbb R\end{equation}

\noindent (see \cite{Hel}, p. 128).  By taking into account part (i) of Lemma \ref{egal}, relation (\ref{Ti}) as well as identity (\ref{id3}), we obtain 

\begin{eqnarray*}T_iX_{i+1}(t)&=&T_i^N(t)X_{i+1}=T_{i+1}^N(t)T_i^{i+1}(t)X_{i+1}\\
&=&T_{i+1}^N(t)\operatorname{Ad}(\phi_{X_{i+1}}(-t))X_{i+1}=T_{i+1}^N(t)X_{i+1}\\
&=&T_{i+1}X_{i+1}(t),\end{eqnarray*}

\noindent which concludes the proof of part (i).  For parts (ii), (iii) and (iv), we firstly write equations (\ref{VE2}) as 

\begin{eqnarray}\label{peqnar}X_1&=&Y_1+\frac{\lambda_2}{\lambda_1-\lambda_2}X_2, \\
&& \nonumber\\
\label{peqnar1}X_i&=&\frac{\lambda_{i-1}-\lambda_i}{\lambda_{i-1}}Y_i+\frac{\lambda_{i+1}(\lambda_{i-1}-\lambda_i)}{\lambda_{i-1}(\lambda_i-\lambda_{i+1})}X_{i+1}, \quad 2\leq i\leq N-1, \\
&&\nonumber \\
\label{peqnar2}X_N&=&\frac{\lambda_{N-1}-\lambda_N}{\lambda_{N-1}}Y_N.
\end{eqnarray}

By applying $T_i$ on the expressions (\ref{peqnar}) - (\ref{peqnar2}) for $X_i$, $i=1,\dots ,N$, by using part (i) of the lemma as well as the fact that $T_i(X+Y)=T_iX+T_iY$ for any $X,Y\in \fr{g}$, we obtain 

\begin{eqnarray}\label{peqnar3}T_1X_1&=&T_1Y_1+\frac{\lambda_2}{\lambda_1-\lambda_2}T_2X_2, \\
&& \nonumber\\
\label{peqnar4}T_iX_i&=&\frac{\lambda_{i-1}-\lambda_i}{\lambda_{i-1}}T_iY_i+\frac{\lambda_{i+1}(\lambda_{i-1}-\lambda_i)}{\lambda_{i-1}(\lambda_i-\lambda_{i+1})}T_{i+1}X_{i+1}, \quad 2\leq i\leq N-1,  \\
& & \nonumber\\
\label{peqnar5}T_NX_N&=&\frac{\lambda_{N-1}-\lambda_N}{\lambda_{N-1}}T_NY_N.
\end{eqnarray}

By using equations (\ref{peqnar3}) - (\ref{peqnar5}) recursively, we obtain

\begin{eqnarray}\label{eqnar2}
T_1X_1&=&T_1Y_1+\frac{\lambda_2}{\lambda_1-\lambda_2}T_2X_2=\nonumber \\
&& \nonumber \\
&=&T_1Y_1+\dots +\frac{\lambda_i}{\lambda_1}T_iY_i+\frac{\lambda_{i}\lambda_{i+1}}{\lambda_{1}(\lambda_i-\lambda_{i+1})}T_{i+1}X_{i+1} \quad (2\leq i\leq N-1)\nonumber \\
&& \nonumber \\
&=&T_1Y_1+\frac{\lambda_2}{\lambda_1}T_2Y_2+\dots +\frac{\lambda_{N-1}}{\lambda_1}T_{N-1}Y_{N-1}+\frac{\lambda_{N-1}\lambda_{N}}{\lambda_{1}(\lambda_{N-1}-\lambda_{N})}T_{N}X_{N}\nonumber \\
&&\nonumber \\
&=&\frac{1}{\lambda_1}\sum_{k=1}^N\lambda_k T_kY_k. \nonumber
\end{eqnarray}

A similar process shows that $T_iX_i=\frac{(\lambda_{i-1}-\lambda_i)}{\lambda_{i-1}\lambda_i}\sum_{k=i}^N{\lambda_k T_kY_k}$ for $i\geq 2$, which concludes the proof of part (ii).  For part (iii), we use part (ii) to write

\begin{equation}\label{indag}\sum_{k=i}^NT_kX_k=\sum_{k=i}^N\big(\frac{\lambda_{k-1}-\lambda_k}{\lambda_{k-1}\lambda_k}\sum_{j=k}^N\lambda_j T_jY_j\big).\end{equation}

By using induction on $N$, we can change the summation as follows.

\begin{equation}\label{soult}\sum_{k=i}^N\big(\frac{\lambda_{k-1}-\lambda_k}{\lambda_{k-1}\lambda_k}\sum_{j=k}^N\lambda_j T_jY_j\big)=\sum_{k=i}^N\big(\sum_{j=i}^k\frac{\lambda_{j-1}-\lambda_j}{\lambda_{j-1}\lambda_j}\big)\lambda_kT_kY_k.\end{equation} 

By taking into account Equation (\ref{soult}), Equation (\ref{indag}) yields

\begin{equation*}\sum_{k=i}^NT_kX_k=\sum_{k=i}^N\big(\sum_{j=i}^k\frac{\lambda_{j-1}-\lambda_j}{\lambda_{j-1}\lambda_j}\big)\lambda_kT_kY_k=\sum_{k=i}^N\frac{\lambda_{i-1}-\lambda_k}{\lambda_{i-1}\lambda_k}\lambda_kT_kY_k=\sum_{k=i}^N\frac{\lambda_{i-1}-\lambda_k}{\lambda_{i-1}}T_kY_k,
\end{equation*}

\noindent which concludes the proof of part (iii).  Finally, for part (iv), we use parts (ii) and (iii) to obtain

\begin{eqnarray*}\sum_{k=1}^NT_kX_k&=&T_1X_1+\sum_{k=2}^NT_kX_k=\frac{1}{\lambda_1}\sum_{k=1}^N\lambda_k T_kY_k+\frac{1}{\lambda_1}\sum_{k=2}^N(\lambda_{1}-\lambda_k)T_kY_k\nonumber \\
&=&\sum_{k=1}^NT_kY_k,
\end{eqnarray*}

\noindent which concludes the proof of part (iv) of Lemma \ref{claim}.\qed\\

 \noindent \emph{Proof of Lemma \ref{claim2}}.  We set $M_i= \bigoplus_{k=i}^N\fr{m}_i$.  Condition (\ref{integ}) of Theorem \ref{main11} implies that $[\fr{m}_i,M_j]\subseteq \fr{m}_i$ if $i<j$.  Hence

\begin{equation}\label{logar}\operatorname{ad}^k(M_j)\fr{m}_i\subseteq \fr{m}_i\quad \makebox{for any} \quad k\in \mathbb Z^+, \quad \makebox{and} \quad i<j.\end{equation}

 Moreover, relations (\ref{VE2}) and the fact that $Y_i\in \fr{m}_i$, $i=1,\dots,N$, imply that 

\begin{equation}\label{ginger}X_i\in M_i,\quad \makebox{for any} \quad i=1,\dots,N.\end{equation}

Relation (\ref{ginger}), along with identity (\ref{id3}) and relation (\ref{logar}), implies that 

\begin{equation}\label{fres}\operatorname{Ad}(\phi_{X_j}(-t))\fr{m}_i\subseteq \sum_{k=0}^{\infty}\frac{1}{k!}\operatorname{ad}^k(M_j)\fr{m}_i\subseteq \fr{m}_i,\end{equation}

\noindent for any $i,j$ with $1\leq i <j \leq N$.  By taking into account the fact that $Y_i\in \fr{m}_i$, relations (\ref{Ti}) as well as relation (\ref{fres}), we obtain 

\begin{eqnarray*}T_iY_i(t)&=&T_i^N(t)Y_i= \operatorname{Ad}((\prod_{k=i+1}^N{\phi_{X_k}(t)})^{-1})Y_i\nonumber \\
&\in&\operatorname{Ad}(\phi_{X_N}(-t))\operatorname{Ad}(\phi_{X_{N-1}}(-t))\cdots \operatorname{Ad}(\phi_{X_{i+1}}(-t))\fr{m}_i \subseteq  \fr{m}_i,
\end{eqnarray*}

\noindent for any $i=1,\dots, N$, which shows part (i).  Part (ii) holds trivially if $k=i$.  If $k>i$, part (ii) follows immediately from part (i) and Condition (\ref{integ}).\qed\\

Next, we proceed to conclude the proof of Theorem \ref{main11} by verifying Equation (\ref{gelane}).  Firstly, we consider the trivial case $N=1$.  In that case we have $X_1=Y_1$ and $T_1X_1=Y_1$.  Moreover, the space $\fr{m}$ is a $\lambda_1$-eigenspace of $A_{\mu}$.   Thus the left-hand side of the geodesic equation (\ref{gelane}) becomes
  
  \begin{eqnarray*}A_{\mu}[T_1X_1, T_1X_1]_{\fr{m}}-[A_{\mu}(T_1X_1)_{\fr{m}},T_1X_1]&=&-[A_{\mu}(X_1)_{\fr{m}},X_1]\nonumber \\
  &=&-[A_{\mu}Y_1,Y_1]=-[\lambda_1 Y_1,Y_1]=0,\nonumber
\end{eqnarray*}

\noindent which concludes the proof of Theorem \ref{main11} for the trivial case.  We will henceforth assume that $N\geq 2$.  We will rewrite the geodesic equation (\ref{gelane}) in terms of $T_iY_i$.  We denote by $L_1$ the first term on the left-hand side of Equation (\ref{gelane}), i.e. $L_1=\sum_{i=1}^NA_{\mu}\big[{T_iX_i},\sum_{k=i}^N{T_kX_k}\big]_{\fr{m}}$.  By using parts (ii) - (iv) of Lemma \ref{claim}, the term $L_1$ becomes 

\begin{eqnarray}L_1&=&A_{\mu}\big[T_1X_1,\sum_{k=1}^NT_kX_k\big]_{\fr{m}}+\sum_{i=2}^NA_{\mu}\big[T_iX_i,\sum_{k=i}^NT_kX_k\big]_{\fr{m}}\nonumber \\
&&\nonumber \\
&=&\frac{1}{\lambda_1}A_{\mu}\big[\sum_{j=1}^N\lambda_j T_jY_j,\sum_{k=1}^NT_kY_k\big]_{\fr{m}}\nonumber \\
&&+\sum_{i=2}^N(\frac{\lambda_{i-1}-\lambda_i}{\lambda_{i-1}\lambda_i})A_{\mu}\bigg[\sum_{j=i}^N\lambda_j T_jY_j,\frac{1}{\lambda_{i-1}}\sum_{k=i}^N(\lambda_{i-1}-\lambda_k)T_kY_k\bigg]_{\fr{m}}\nonumber \\
&& \nonumber \\
&=&\frac{1}{\lambda_1}A_{\mu}\big[\sum_{j=1}^N\lambda_j T_jY_j,\sum_{k=1}^NT_kY_k\big]_{\fr{m}}+\sum_{i=2}^N(\frac{\lambda_{i-1}-\lambda_i}{\lambda_{i-1}\lambda_i})A_{\mu}\big[\sum_{j=i}^N\lambda_j T_jY_j,\sum_{k=i}^NT_kY_k\big]_{\fr{m}}\nonumber \\
&& \nonumber \\
&=&\frac{1}{\lambda_1}\sum_{j\leq k=1}^N(\lambda_j-\lambda_k)A_{\mu}\big[T_jY_j,T_kY_k\big]_{\fr{m}}+\sum_{i=2}^N\bigg(\frac{\lambda_{i-1}-\lambda_i}{\lambda_{i-1}\lambda_i}\sum_{j\leq k=i}^N(\lambda_j-\lambda_k)A_{\mu}\big[T_jY_j,T_kY_k\big]_{\fr{m}}\bigg).\nonumber\\
&& \label{gelan3}
\end{eqnarray}

We set 

\begin{eqnarray}c_1&:=&\frac{1}{\lambda_1},\quad c_i:=\frac{\lambda_{i-1}-\lambda_i}{\lambda_{i-1}\lambda_i}, \quad i=2,\dots,N,\label{soulextr0}\\
 d_{jk}&:=&(\lambda_j-\lambda_k)A_{\mu}\big[T_jY_j,T_kY_k\big]_{\fr{m}}, \quad1\leq j\leq k\leq N,\label{soulextr1}\end{eqnarray} 
 
\noindent so that Equation (\ref{gelan3}) is written as

\begin{equation}\label{soulextr2}L_1=\sum_{i=1}^{N}(\sum_{j\leq k=i}^{N}c_id_{jk})=\sum_{i\leq j\leq k=1}^{N}c_id_{jk}.\end{equation}

By using induction on $N$, we can change the summation as follows.

\begin{equation}\label{soulextr4}\sum_{i\leq j\leq k=1}^{N}c_id_{jk}=\sum_{j\leq k=1}^N(\sum_{i=1}^jc_i)d_{jk}.\end{equation}

Then by using equations (\ref{soulextr0}) - (\ref{soulextr4}), the first term on the left-hand side of Equation (\ref{gelane}) finally becomes

\begin{eqnarray}\label{gelan4}\sum_{i=1}^NA_{\mu}\big[{T_iX_i},\sum_{k=i}^N{T_kX_k}\big]_{\fr{m}}&=&L_1=\sum_{j\leq k=1}^N(\sum_{i=1}^jc_i)d_{jk}=\sum_{j\leq k=1}^N\frac{1}{\lambda_j}d_{jk}\nonumber \\
&=&\sum_{j\leq k=1}^N\frac{\lambda_j-\lambda_k}{\lambda_j}A_{\mu}[T_jY_j,T_kY_k]_{\fr{m}}.
\end{eqnarray}

Subsequently, by using parts (iii) and (iv) of Lemma \ref{claim}, the second term on the left-hand side of the geodesic equation (\ref{gelane}) becomes

\begin{equation}\label{gelan5}-\big[A_{\mu}\sum_{i=1}^N({T_iX_i)}_{\fr{m}},\sum_{k=1}^N{T_kX_k}\big]=-\big[A_{\mu}\sum_{i=1}^N({T_iY_i)}_{\fr{m}},\sum_{k=1}^N{T_kY_k}\big].\end{equation}

By summing equations (\ref{gelan4}) and (\ref{gelan5}), we deduce that the geodesic equation (\ref{gelane}) is equivalent to

\begin{equation}\label{gelan6}
\sum_{j\leq k=1}^N\frac{\lambda_j-\lambda_k}{\lambda_j}A_{\mu}[T_jY_j,T_kY_k]_{\fr{m}}-\big[A_{\mu}\sum_{i=1}^N({T_iY_i)}_{\fr{m}},\sum_{k=1}^N{T_kY_k}\big]=0.\end{equation}

In order to verify Equation (\ref{gelan6}) and conclude the proof of Theorem \ref{main11}, the final step is to evaluate the metric endomorphism $A_{\mu}$ on the vectors $\big[{T_jY_j},{T_kY_k}\big]_{\fr{m}}$ and $T_iY_i$.  By using Lemma \ref{claim2} as well as the fact that the subspaces $\fr{m}_i$ of $\fr{m}$ are $\lambda_i$-eigenspaces of the endomorphism $A_{\mu}$, the left-hand side $\hat{L}$ of Equation (\ref{gelan6}) becomes  

\begin{eqnarray*}
\hat{L}&=&\sum_{j\leq k=1}^N\frac{\lambda_j-\lambda_k}{\lambda_j}A_{\mu}[T_jY_j,T_kY_k]_{\fr{m}}-\big[A_{\mu}\sum_{i=1}^N({T_iY_i)}_{\fr{m}},\sum_{k=1}^N{T_kY_k}\big]\nonumber \\
&=&\sum_{j\leq k=1}^N\frac{\lambda_j-\lambda_k}{\lambda_j}\lambda_j[T_jY_j,T_kY_k]-\big[\sum_{i=1}^N\lambda_i{T_iY_i},\sum_{k=1}^N{T_kY_k}\big]\nonumber \\
&=&\sum_{j\leq k=1}^N(\lambda_j-\lambda_k)[T_jY_j,T_kY_k]-\sum_{i=1}^N\sum_{k=1}^N[\lambda_i{T_iY_i},T_kY_k]\nonumber \\
&=&\sum_{j\leq k=1}^N(\lambda_j-\lambda_k)[T_jY_j,T_kY_k]-\sum_{i\leq k=1}^N(\lambda_i-\lambda_k)[T_iY_i,T_kY_k]=0,\nonumber 
\end{eqnarray*}
 
\noindent  which concludes the proof of Theorem \ref{main11}.

\section{Metrics whose geodesics are orbits of products of one-parameter subgroups}\label{appl}

In this section we prove Theorem \ref{main3} by constructing a general class of pseudo-Riemannian metrics $\mu(\lambda_1,\lambda_2,\dots,\lambda_N)$ whose geodesics are described by Theorem \ref{main11}.  By using the aforementioned construction, we present applications of Theorem \ref{main3}.  Moreover, we give an interpretation of the central Theorem \ref{main11}, for certain two parameter metrics, by means of a result of Ziller.  To prove Theorem \ref{main3}, we need the following structural result for naturally reductive spaces.  The result is due to Kostant.  We will state a reformulation of the result by D'Atri and Ziller.

\begin{theorem}\label{Kost}\emph{(\cite{Kos}, \cite{Da-Zi})}  Let $K/L$ be a homogeneous space with $K$ acting effectively on $K/L$.  Let $\fr{k}$, $\fr{l}$ be the Lie algebras of $K$, $L$ respectively.  Assume that $\mu_0$ is a $K$-invariant Riemannian naturally reductive metric with respect to the decomposition $\fr{k}=\fr{l}\oplus \fr{m}$, and let $\langle \ ,\ \rangle_0$ be the corresponding inner product in $\fr{m}$.  Then the space $\fr{g}:=\fr{m}+[\fr{m},\fr{m}]$ is an ideal in $\fr{k}$ and has the following properties:\\

\noindent (i) The corresponding connected Lie subgroup $G$ of $K$, with Lie algebra $\fr{g}$, acts transitively on $K/L$, making the space $K/L$ diffeomorphic to $G/H$, where $H=G\cap L$.\\

\noindent (ii) There exists a unique $\operatorname{Ad}(G)$-invariant, symmetric and non-degenerate bilinear form $Q$ on $\fr{g}$ such that 

\begin{equation*}\label{eqkost}Q(\fr{h}:=\fr{g}\cap\fr{l},\fr{m})=\{0\}\quad \makebox{and} \quad \left.Q\right|_{\fr{m}\times \fr{m}}=\langle \ ,\ \rangle_0.\end{equation*}

In other words, there exists a homogeneous space $G/H$, diffeomorphic to $K/L$, such that the metric $\mu_0$ in $G/H$ is induced from an $\operatorname{Ad}(G)$-invariant form in the sense of Definition \ref{indic}.
\end{theorem}

\subsection{Proof of Theorem \ref{main3}}\label{app}

Under the assumptions of Theorem \ref{main3}, there exists a homogeneous space $K/L$ diffeomorphic to $M$, endowed with a $K$-invariant naturally reductive Riemannian metric $\mu_0$.  By virtue of Theorem \ref{Kost}, we can find a Lie subgroup $G$ of $K$ and an $\operatorname{Ad}(G)$-invariant, symmetric, non-degenerate bilinear form $Q$ such that properties (i) and (ii) hold.  Let $H_1,\dots,H_{N-1},H_N$ be Lie groups such that $H\subset H_1\subset \cdots \subset H_{N-1}\subset H_N:=G$, and let $\fr{h}_1,\dots,\fr{h}_{N-1},\fr{h}_N$ be their corresponding Lie algebras with $\fr{h}\subset \fr{h}_1\subset \cdots \subset \fr{h}_{N-1}\subset \fr{h}_N:=\fr{g}$.  Using the same arguments as in subsection \ref{cf}, we take into account the fact that $\left.Q\right|_{\fr{m}\times \fr{m}}$ is positive definite, the $Q$-orthogonality between $\fr{h}$ and $\fr{m}$ as well as the $\operatorname{Ad}(G)$-invariance of $Q$ to obtain a $Q$-orthogonal decomposition

\begin{equation}\label{votsala}\fr{g}=\fr{h}\oplus \fr{m}, \quad \makebox{with} \quad \operatorname{Ad}(H)\fr{m}\subseteq \fr{m}.\end{equation}

  We set $\fr{h}_0:=\fr{h}$ and we define subspaces $\fr{m}_1,\dots,\fr{m}_N$ of $\fr{m}$ via the $Q$-orthogonal decompositions
 
 \begin{equation}\label{miai}\fr{h}_{i}=\fr{h_{i-1}}\oplus\fr{m}_{N-i+1},\quad i=1,\dots,N.\end{equation}

The sums (\ref{miai}) are direct due to the fact that $\left.Q\right|_{\fr{m}_i\times \fr{m}_i}$ is positive definite for $i=1,\dots,N$.  We endow $\bigoplus_{i=1}^N{\fr{m}_i}$ with the $N$-parameter family of forms defined by

\begin{equation}\label{formss}\langle \ ,\ \rangle=\left.\lambda_1Q\right|_{\fr{m}_1\times \fr{m}_1}+\cdots + \left.\lambda_NQ\right|_{\fr{m}_N\times \fr{m}_N}, \quad \lambda_i\in \mathbb R^*.\end{equation}

 We will prove that the forms (\ref{formss}) correspond to the desired $N$-parameter family of $G$-invariant metrics $\mu(\lambda_1,\dots,\lambda_N)$ satisfying Theorem \ref{main11}.  To this end, we will firstly show the following.

\begin{lemma}\label{claim3}The subspaces $\fr{m}_i$, defined by the decompositions (\ref{miai}), satisfy the following relations.\\
 
\noindent  (i) $\fr{m}=\bigoplus_{i=1}^N{\fr{m}_i}$.\\

\noindent (ii) $\operatorname{Ad}(H_{i-1})\fr{m}_{N-i+1}\subseteq \fr{m}_{N-i+1}$,\quad for any \quad $i=1,\dots,N$.\\
 
\noindent (iii) $\operatorname{Ad}(H)\fr{m}_i\subseteq \fr{m}_i$, \quad for any \quad $i=1,\dots,N$.\\
 
\noindent (iv) $[\fr{m}_i,\fr{m}_j]\subseteq \fr{m}_i$\quad for any \quad i,j \quad with \quad $1\leq i<j\leq N$.\\
\end{lemma}

\begin{proof}

By using relations (\ref{miai}) recursively, we obtain 

\begin{equation*}\fr{g}=\fr{h}_N=\fr{h}_{N-1}\oplus\fr{m}_1=\fr{h}_{N-2}\oplus \fr{m}_2\oplus \fr{m}_1=\dots =\fr{h}\oplus\fr{m}_N\oplus\cdots \oplus \fr{m}_1,\end{equation*}

\noindent which, along with decomposition (\ref{votsala}), proves part (i).  For part (ii), decompositions (\ref{miai}) imply that $\fr{m}_{N-i+1}\subseteq \fr{h}_{i}$, therefore 

\begin{equation}\label{skulo}\operatorname{Ad}(H_{i-1})\fr{m}_{N-i+1}\subseteq \operatorname{Ad}(H_{i})\fr{m}_{N-i+1}\subseteq \operatorname{Ad}(H_{i})\fr{h}_{i}\subseteq \fr{h}_i=\fr{h}_{i-1}\oplus \fr{m}_{N-i+1}.\end{equation}

On the other hand, the $\operatorname{Ad}(G)$-invariance of $Q$ implies that

\begin{equation*}Q(\operatorname{Ad}(H_{i-1})\fr{m}_{N-i+1},\fr{h}_{i-1})\subseteq Q(\fr{m}_{N-i+1},\operatorname{Ad}(H_{i-1})\fr{h}_{i-1})\subseteq Q(\fr{m}_{N-i+1},\fr{h}_{i-1})\subseteq \{0\}.\end{equation*}

Therefore $\operatorname{Ad}(H_{i-1})\fr{m}_{N-i+1}$ is orthogonal to $\fr{h}_{i-1}$ which, along with relation (\ref{skulo}), verifies part (ii).  Part (iii) follows immediately from part (ii), after taking into account that $H\subset H_i$ for any $i=1,\dots,N$.  Finally, if $i<j$, the inclusions $\fr{m}_j\subset\fr{h}_{N-j+1}\subset \fr{h}_{N-i}$, along with part (ii), imply that

\begin{equation*}[\fr{m}_i,\fr{m}_j]\subseteq [\fr{m}_i,\fr{h}_{N-i}]\subseteq \fr{m}_i,\end{equation*}

\noindent which concludes the proof of part (iv).\end{proof}

We proceed to conclude the proof of Theorem \ref{main3}.  Part (i) of Lemma \ref{claim3} implies that the forms (\ref{formss}) are defined on $\fr{m}$.  Since $\lambda_i\in \mathbb R^*$, the forms are non-degenerate.  Moreover, part (iii) asserts that the subspaces $\fr{m}_i$ are $\operatorname{Ad}(H)$-invariant.  Therefore, the forms (\ref{formss}) correspond to a family of $G$-invariant pseudo-Riemannian metrics $\mu=\mu(\lambda_1,\dots,\lambda_N)$ in $G/H$.  We set $\langle \ ,\ \rangle_0:=\left.Q\right|_{\fr{m}\times \fr{m}}$.  Then the $N$-parameter family of operators $A_{\mu}:\fr{m}\rightarrow \fr{m}$ defined by

\begin{equation}\label{mtc}A_{\mu}=A_{\mu}(\lambda_1,\dots,\lambda_N):=\left.\lambda_1\operatorname{Id}\right|_{\fr{m}_1}+\cdots + \left.\lambda_N\operatorname{Id}\right|_{\fr{m}_N}, \quad \lambda_i\in \mathbb R^*,\end{equation}

\noindent satisfy Equation (\ref{metrop}).  Therefore, $A_{\mu}$ is the metric operator corresponding to $\mu$, and\\
 $\fr{m}_1,\dots,\fr{m}_N$ are the eigenspaces of $A_{\mu}$, corresponding to the eigenvalues $\lambda_1,\dots,\lambda_N$.  Finally, by virtue of part (iv) of Lemma \ref{claim3}, the eigenspaces satisfy the algebraic condition \ref{integ} of Theorem \ref{main11}.  Hence the metrics $\mu(\lambda_1,\dots,\lambda_N)$ satisfy Theorem \ref{main11}, thus concluding the proof of Theorem \ref{main3}.\\
  In the special case where $G/H$ is a homogeneous space with $G$ compact, we may choose $Q$ to be an $\operatorname{Ad}$-invariant inner product on $\fr{g}$.  In that case the arguments in this proof apply for $G/H$.  Therefore the conclusion of Theorem \ref{main3} holds for any homogeneous space $G/H$ with $G$ compact.

\subsection{Examples of the metrics $\mu(\lambda_1,\dots ,\lambda_N)$}
\begin{example}
We consider the metrics $\mu(\lambda_1,\dots,\lambda_N)$ constructed from the operators defined in Equation (\ref{mtc}).  In the case where the Lie subgroups $H_0,H_1,\dots,H_{N-1},H_{N}$, with $H=H_0\subset H_1\subset \cdots H_{N-1}\subset H_N=G$, are closed subgroups of $G$, we can define the manifolds $H_{i}/H_{i-1}$, $i=1,\dots,N$.  Then by virtue of part (ii) of Lemma \ref{claim3}, the subspaces $\fr{m},\fr{m}_1,\dots,\fr{m}_N$ of $\fr{g}$, defined by the $Q$-orthogonal decompositions (\ref{votsala}) and (\ref{miai}), can be identified with the tangent spaces $T_{eH}(G/H), T_{eH_{N-1}}(H_N/H_{N-1}),\dots, T_{eH_0}(H_1/H_0)$ respectively.  Then part (i) of Lemma \ref{claim3} implies the decomposition

\begin{equation*}T_{eH}(G/H)=T_{eH_{0}}(H_1/H_{0})\oplus \cdots  \oplus T_{eH_{N-1}}(H_N/H_{N-1}).\end{equation*}

In this regard, the metrics $\mu(\lambda_1,\dots,\lambda_N)$ are deformations of the naturally reductive metric $\left.Q\right|_{\fr{m}\times \fr{m}}$ on $G/H$, along the manifolds $H_i/H_{i-1}$, $i=1,\dots,N$.  When $G$ is compact, well-known examples of such metrics include the \emph{Cheeger deformations} of a normal metric (the metric induced from a bi-invariant metric in $G$), along the fibers of a homogeneous fibration $K/H\rightarrow G/H\rightarrow G/K$, where $H\subset K\subset G$.  These metrics where introduced in \cite{Ch} within a more general context.
\end{example}

\begin{example}
There exist numerous examples of spaces $G/H$ such that any $G$-invariant pseudo-Riemannian metric has the form $\mu(\lambda_1,\lambda_2)$.  This is the case when $G$ is a compact and simple Lie group, $H$ is a connected closed subgroup of $G$ which is non-maximal as a Lie subgroup, and the isotropy representation $\rho:H\rightarrow \operatorname{Gl}(\fr{m})$, defined by $\rho(h)X=\operatorname{Ad}(h)(X)$, induces the decomposition of $\fr{m}$ into exactly two inequivalent and irreducible submodules $\fr{m}_1,\fr{m}_2$.  Indeed, for any metric endomorphism $A$ on $\fr{m}=\fr{m}_1\oplus \fr{m}_2$, the $\operatorname{Ad}(H)$-invariance of the eigenspaces of $A$, along with the pairwise inequivalence and the irreducibility of $\fr{m}_1,\fr{m}_2$, impose the condition $A=\left.\lambda_1\operatorname{Id}\right|_{\fr{m}_1}+\left.\lambda_2\operatorname{Id}\right|_{\fr{m}_2}$.  Moreover the non-maximality of $H$ implies that $\fr{m}_1,\fr{m}_2$ are induced from a subgroup series of the form $H\subset K\subset G$, therefore, any metric has the form $\mu(\lambda_1,\lambda_2)$.  An example is the space $G_2/U(2)$, where any metric is induced from the sequence $U(2)\subset SO(4)\subset G_2$.  We refer to \cite{Di-Ke} for the classification of the spaces $G/H$ where $G$ is a compact and simple Lie group and the isotropy representation has exactly two irreducible summands. 
\end{example}
\subsection{An interpretation of Theorem \ref{main11} for $N=2$}\label{applll}
For certain two-parameter metrics $\mu(\lambda_1,\lambda_2)$ in $G/H$, the fact that the geodesics deform to orbits of a product of two one-parameter subgroups is related to the natural reductivity of the metrics $\mu(\lambda_1,\lambda_2)$ with respect to a subgroup of $G\times G$.  In particular, let $G/H$ be a homogeneous space with $H$ connected, and $G$ compact and semisimple.  Consider a Lie subgroup series $H\subset K \subset G$ such that $K$ is connected and $H$ is normal in $K$.  Then according to the construction in subsection \ref{app}, we obtain a $Q$-orthogonal decomposition $\fr{g}=\fr{h}\oplus \fr{m}_1\oplus \fr{m}_2$ with respect to an $\operatorname{Ad}$-invariant form $Q$ in $\fr{g}$.  We may choose $Q$ to be the negative of the Killing form of $\fr{g}$.  Here $\fr{h}\oplus \fr{m}_2$ is the Lie algebra of the intermediate subgroup $K$.  Since $Q$ is $\operatorname{Ad}$-invariant and $H$ is normal in $K$, we can show that $[\fr{h},\fr{m}_2]=\{0\}$ and $[\fr{m}_2,\fr{m}_2]\subset \fr{m}_2$. \\
Ziller showed that under the aforementioned relations, the metrics $\mu=\mu(\lambda_1,\lambda)$ corresponding to the endomorphisms 

\begin{equation*}A=\left.\lambda_1\operatorname{Id}\right|_{\fr{m}_1}+\left.\lambda_2\operatorname{Id}\right|_{\fr{m}_2},\quad \lambda_1,\lambda_2>0,\end{equation*}

\noindent are $G\times K$-invariant and naturally reductive (see \cite{Zi}, Section 3).  By virtue of Proposition \ref{oneiln}, the geodesics of $(G/H,\mu)$ are orbits of single one-parameter subgroups of $G\times K$.  However, any one-parameter subgroup of $G\times K$ has the form $\phi_{X_1}(t)\phi_{X_2}(t)$, where $\phi_{X_1}(t),\phi_{X_2}(t)$ are one-parameter subgroups of $G,K\subset G$ respectively.  Hence the geodesics of $(G/H,\mu)$ have the form $\phi_{X_1}(t)\phi_{X_2}(t)H$, which in turn agrees with the result of Theorem \ref{main11} for the metrics $\mu=\mu(\lambda_1,\lambda_2)$.

\end{document}